\documentclass[reqno, 11pt]{amsart}

\usepackage{amsfonts, amsthm, amsmath, amssymb}
\usepackage{hyperref}
\hypersetup{colorlinks=false}

\usepackage{amsmath,amsthm,amsfonts,amssymb,bm,graphicx,xparse}

\usepackage[margin=3.5cm]{geometry}

\makeatletter
\NewDocumentCommand{\sump}{e{_}}
 {%
  \DOTSB
  \mathop{\IfNoValueTF{#1}{\sump@{}}{\sump@{#1}}}%
  \nolimits
 }
\newcommand{\sump@}[1]{\mathpalette\sump@@{#1}}
\newcommand{\sump@@}[2]{%
  \ifx#1\displaystyle
    {\sump@display{#2}}%
  \else
    \sum@\nolimits^*_{#2}%
  \fi
}
\newcommand{\sump@display}[1]{%
  \sbox\z@{$\m@th\displaystyle\sum@\nolimits'$}%
  \sbox\tw@{$\m@th\displaystyle\sum@\limits_{#1}$}%
  \sbox\@tempboxa{$\m@th\displaystyle'$}
  \mathop{\sum@\nolimits^* \kern-\wd\@tempboxa}\limits_{#1}%
  \ifdim\wd\z@>\wd\tw@
    \kern\dimexpr\wd\z@-\wd\tw@\relax
  \fi
}
\makeatother

\RequirePackage{mathrsfs} \let\mathcal\mathscr
  
\renewcommand{\leq}{\leqslant}
\renewcommand{\le}{\leqslant}
\renewcommand{\geq}{\geqslant}
\renewcommand{\ge}{\geqslant}

\usepackage{comment}
\usepackage{graphics}
\usepackage{aliascnt}

\usepackage{enumerate}
\usepackage{amsmath}
\usepackage{amsfonts}
\usepackage{amssymb}
\usepackage{amsthm}
\usepackage{comment}
\usepackage{mathtools}

\usepackage{amsmath,amsthm,amsfonts,amssymb,bm,graphicx, float, subfigure}

\usepackage{enumitem}

\makeatletter
\@namedef{subjclassname@2020}{\textup{2020} Mathematics Subject Classification}
\makeatother

\newtheorem{theorem}{Theorem}[section]
\newtheorem{Theorem}[theorem]{Theorem}

\newtheorem{lemma}[theorem]{Lemma}

\newtheorem{prop}[theorem]{Proposition}
\newtheorem{Prop}[theorem]{Proposition}

\theoremstyle{definition}

\newtheorem{remark}[theorem]{Remark}

\theoremstyle{definition}

\newtheorem{Definition}[theorem]{Definition}
\newtheorem{definition}[theorem]{Definition}

\renewcommand{\geq}{\geqslant}
\renewcommand{\leq}{\leqslant}
\renewcommand{\mod}{\mathrm{mod}\,}

\newcommand{\norm}[1]{\left\lVert#1\right\rVert}

\newcommand{\lcm}{\mathrm{lcm}}

\newcommand{\eps}{\varepsilon}

\newcommand{\FF}{\mathbb{F}}
\newcommand{\OO}{\mathcal{O}}
\newcommand{\RR}{\mathbb{R}}
\newcommand{\BB}{\mathcal{B}}
\newcommand{\AAA}{\mathbb{A}}

\newcommand{\mm}{\mathbf{m}}

\newcommand{\cc}{\mathbf{c}}
\newcommand{\ee}{\mathbf{e}}
\newcommand{\nn}{\mathbf{n}}
\newcommand{\uu}{\mathbf{u}}
\newcommand{\sss}{\mathbf{s}}
\newcommand{\xx}{\mathbf{x}}
\newcommand{\pa}{\partial}
\newcommand{\yy}{\mathbf{y}}

\newcommand{\CC}{\mathbb{C}}
\newcommand{\CCC}{\mathcal{C}}
\newcommand{\QQ}{\mathbb{Q}}
\newcommand{\ZZ}{\mathbb{Z}}
\newcommand{\NN}{\mathbb{N}}
\newcommand{\NNN}{\mathbf{N}}

\numberwithin{equation}{section}

\makeatletter
\def\Ddots{\mathinner{\mkern1mu\raise\p@
\vbox{\kern7\p@\hbox{.}}\mkern2mu
\raise4\p@\hbox{.}\mkern2mu\raise7\p@\hbox{.}\mkern1mu}}
\makeatother
\makeatletter
\DeclareRobustCommand\widecheck[1]{{\mathpalette\@widecheck{#1}}}
\def\@widecheck#1#2{%
    \setbox\z@\hbox{\m@th$#1#2$}%
    \setbox\tw@\hbox{\m@th$#1%
       \widehat{%
          \vrule\@width\z@\@height\ht\z@
          \vrule\@height\z@\@width\wd\z@}$}%
    \dp\tw@-\ht\z@
    \@tempdima\ht\z@ \advance\@tempdima2\ht\tw@ \divide\@tempdima\thr@@
    \setbox\tw@\hbox{%
       \raise\@tempdima\hbox{\scalebox{1}[-1]{\lower\@tempdima\box
\tw@}}}%
    {\ooalign{\box\tw@ \cr \box\z@}}}
\makeatother

\begin{document}

\title[Liouville, von Mangoldt and norm forms at random binary forms]{Liouville function, von Mangoldt function and norm forms at random binary forms}

\author{Yijie Diao}
\address{ISTA\\
Am Campus 1\\
3400 Klosterneuburg\\
Austria}
\email{yijie.diao@ist.ac.at}

\begin{abstract}
We analyze the average behavior of various arithmetic functions at the values of degree $d$ binary forms ordered by height, with probability $1$. This approach yields averaged versions of the Chowla conjecture and the Bateman--Horn conjecture for random binary forms. Furthermore, we show that the rational Hasse principle holds for almost all Ch\^atelet varieties defined by a fixed norm form of degree $e$ and by varying binary forms of fixed degree $d$, provided $e$ divides $d$. This proves an average version of a conjecture of Colliot-Th\'el\`ene.

\end{abstract}


\maketitle

\tableofcontents

\section{Introduction}

Browning, Sofos, and Ter\"av\"ainen \cite{BST} recently studied the average behavior of arithmetic functions at random polynomials. They proved averaged versions of the Bateman--Horn conjecture, the polynomial Chowla conjecture, and the integral Hasse principle for norm form equations.
The fundamental tool in their work links equidistribution of an arithmetic function in arithmetic progressions to its values over almost all polynomials. In this paper, we extend this framework to random binary forms and obtain analogous results. Notably, we prove the rational Hasse principle for $100\%$ of norm form equations.
To rigorously describe the proportion of a certain class of binary forms, we will first define \emph{combinatorial cubes}.
\begin{Definition}[Combinatorial cube]
	Let $d \in \ZZ_{\geq 1}$ and $H \in \RR_{\geq 1}$. For any index set $\mathcal{E} \subset \{0,\dots,d\}$ and any integers $\alpha_e \in [-H, H]$ for $e \in \mathcal{E}$, we call the set
	$$\mathcal{C} = \mathcal{C}(H) = \left\{ (c_0, \dots, c_d) \in (\ZZ \cap [-H, H])^{d+1}: c_e = \alpha_e, \forall \, e \in \mathcal{E}  \right\}$$
	a \emph{combinatorial cube} of side length $H$ and dimension $d + 1 - \# \mathcal{E}$. We call $c_d$ the \emph{constant coefficient} of an element $(c_0, \dots, c_d) \in \mathcal{C}$. Given a combinatorial cube $\mathcal{C}$, we say a binary form 
	$$\sum_{i=0}^d c_i x^i y^{d-i}$$
	\emph{has its coefficients in} $\mathcal{C}$ if $(c_0, \dots, c_d) \in \mathcal{C}$.
\end{Definition}

For example, fixing $c_0=1$ and allowing $c_1,\dots,c_d$ to vary over $\ZZ \cap [-H,H]$ gives a combinatorial cube of dimension $d$.

\subsection{Liouville function and von Mangoldt function on random binary forms}
Let $\lambda$ denote the Liouville function, which we extend to all integers by setting $\lambda(0) = 0$ and $\lambda(-n) = \lambda(n)$ for $n > 0$. The polynomial Chowla conjecture \cite{Chowla} states that, for any  $f \in \mathbb{Z}[t]$ not of the form  $cg(t)^2$  with $c \in \mathbb{R}$  and $g \in \mathbb{R}[t]$ , we have
$$\sum_{n \leq x} \lambda(f(n)) = o(x).$$
Recent progress has been made for certain special classes of polynomials \cite{Tao2016, TT-ANT, TT-Duke}. In \cite[Theorem 1.5]{BST}, Browning, Sofos and Ter\"av\"ainen proved the polynomial Chowla conjecture for almost all polynomials, improving on a qualitative result obtained by Ter\"av\"ainen \cite[Theorem 2.11]{Joni}. 

Here, instead of polynomials, we investigate the behavior of the Liouville function at random binary forms with integer coefficients.

\begin{theorem}\label{thm-Chowla}
	Let $d \geq 1, A \geq 1$ and $0 < c < 5/(31d)$ be fixed. There exists a constant $H_0(d,A,c)$ such that the following holds for any $H \geq H_0(d,A,c)$. Let $\mathcal{C} \subset \ZZ^{d+1}$ be a combinatorial cube of side length $H$ and dimension at least $2$. Then, for all but at most $\#\CCC / (\log H)^A$ degree $d$ binary forms $g \in \ZZ[s,t]$ with coefficients in $\CCC$, we have
	$$\sup_{x \in [H^c, 2H^c]} \, \frac{1}{x^2} \ \Bigg| \sum_{u, v \leq x} \lambda(g(u,v)) \Bigg| \leq (\log H)^{-A}.$$
\end{theorem}

Let $\Lambda$ denote the von Mangoldt function, which we extend to all integers by setting $\Lambda(0) = 0$ and $\Lambda(-n) = \Lambda(n)$ for $n > 0$.
The Bateman--Horn conjecture predicts an asymptotic for simultaneous prime values of a tuple of irreducible polynomials, with the leading constant given by the corresponding product of local densities. In a single variable it remains largely unresolved, with the sole exception being the case of linear polynomials, which reduces to Dirichlet's theorem on primes in arithmetic progressions. Consequently, much of the progress has been directed toward addressing analogous problems for polynomials in multiple variables. For $n = 2$, significant breakthroughs include Iwaniec's work~\cite{Iwa-1} on quadratic polynomials, Friedlander and Iwaniec's results~\cite{Fri-Iwa-1} for $x_1^2 + x_2^4$, Heath-Brown's work~\cite{HB2} on the cubic form $x_1^3 + 2x_2^3$, Heath-Brown and Moroz's results~\cite{HB-Mo-1} on binary cubic forms, and the recent contribution by Heath-Brown and Li~\cite{HB-Li-1} concerning $x_1^2 + x_2^4$ with $x_2$ prime. Maynard ~\cite{May} generalized the results of \cite{Fri-Iwa-1} and \cite{HB2} to incomplete norm forms.

We prove the following averaged form of the Bateman--Horn conjecture on the behavior of the von Mangoldt function on binary forms.

\begin{theorem}\label{thm-BH}
	Fix $d, A, r \geq 1$ and $0 < c < 5/(31d)$. There exists $H_0=H_0(d, A, c, r)$ such that the following holds for any $H \geq H_0$. Let $\mathcal{C} \subset \mathbb{Z}^{d+1}$ be a combinatorial cube of side length $H$ and dimension at least $2$, such that the constant coefficient is not fixed to be $0$. Then, for all but at most $(\#\mathcal{C})^r / (\log H)^A$ $r$-tuples of binary forms $g_1, \dots, g_r \in \mathbb{Z}[s, t]$ having degrees $\leq d$ and coefficients in $\mathcal{C}$, we have
\[
\sup_{x \in [H^c, 2H^c]} \frac{1}{x^2} \left| \sum_{m, n \leq x} \Lambda(g_1(m, n)) \cdots \Lambda(g_r(m, n)) - x^2\mathfrak{S}_{g_1, \dots, g_r}(x) \right| \leq (\log H)^{-A},
\]
where
\[
\mathfrak{S}_{g_1, \dots, g_r}(x) = \prod_{p \leq \exp(\sqrt{\log x})} \left( 1 - \frac{1}{p} \right)^{-r} \left( 1 - \frac{\#\{(u, v) \in \mathbb{F}_p^2 : g_1(u, v) \cdots g_r(u, v) = 0 \}}{p^2} \right).
\]
\end{theorem}

\begin{remark}
When the number of variables $n$ is sufficiently large with respect to the degree $d$ of the polynomial, the Hardy--Littlewood circle method becomes particularly useful.
	Destagnol and Sofos showed in \cite{D-S-1} when $n > 2^{d-1}(d-1)$, the Bateman--Horn conjecture is true for non-singular forms in $n$ variables. In \cite{D-S-2}, they showed that when $n > 2^d(d-1)$ the $n$-form Chowla conjecture is true.
\end{remark}

\subsection{Rational Hasse principle for random Ch\^atelet varieties}

Let $K/\QQ$ be a finite extension of degree $e \geq 2$ and fix a $\ZZ$-basis $\boldsymbol{\omega}=\{\omega_1, \dots, \omega_e\}$ for $\OO_K$, the ring of integers of $K$. We define the norm form associated to this basis by
\begin{eqnarray*}
	\mathbf{N}_K(\mathbf{y}) = \mathbf{N}_K(y_1, \dots, y_e) = \mathrm{Norm}_{K/\QQ}(y_1 \omega_1 + \dots + y_e\omega_e). 
\end{eqnarray*}
Let $f(t)$ be a polynomial with integer coefficients.
The {\it Ch\^atelet variety} $X_{K, f}$ is the affine variety $X \subset \AAA_\QQ^{e + 1}$ defined by the equation
\begin{eqnarray}\label{norm-f}
	\mathbf{N}_K(\mathbf{y}) = f(t) \neq 0.
\end{eqnarray}
It is equipped with a dominant morphism $\pi: X \rightarrow \AAA_\QQ^1$.
The basis $\boldsymbol{\omega}$ is fixed throughout, and we suppress it from the notation; a different $\ZZ$-basis gives a $\QQ$-isomorphic affine variety. Since we will order coefficient vectors by height, we emphasize that the height function on $\AAA_\QQ^{e+1}$ implicitly depends on $\boldsymbol{\omega}$, but our results are insensitive to this choice as the basis remains fixed.

Let \( X^c \) be a smooth, projective model associated to \( X \). It is well known that such models may fail to satisfy the Hasse principle or weak approximation. For instance, as observed by Coray (see \cite[Eq.~(8.2)]{CT-S}), one may take \( f(t) = t(t - 1) \) and consider the cubic field extension \( K = \mathbb{Q}(\theta) \), where \( \theta \) is a root of the irreducible polynomial \( y^3 - 7y^2 + 14y - 7 \). In this setting, the rational points \( X^c(\mathbb{Q}) \) are not dense in the set of 7-adic points \( X^c(\mathbb{Q}_7) \). Nevertheless, a conjecture of Colliot-Th\'el\`ene \cite{CT1} proposes that all such failures of Hasse principle for \( X^c \) can be explained via the Brauer--Manin obstruction. Colliot-Th\'el\`ene's conjecture implies that given a natural ordering of the coefficients of $f$, the Hasse principle holds with probability $1$. Browning and Matthiesen \cite{BL} proved Colliot-Th\'el\`ene's conjecture  when $f(t)$ is a product of linear polynomials all defined over $\QQ$, that the Brauer--Manin obstruction is the only obstruction to the Hasse principle and weak approximation.

Colliot-Th\'el\`ene's conjecture extends to more general arithmetic situations, where the base field is a number field \( k \). Progress has been made in verifying this conjecture in various settings, often under specific hypotheses on the extension \( K/k \) and the structure of \( f(t) \). For example, the case of Ch\^atelet surfaces, where \( [K : k] = 2 \) and \( \deg(f(t)) \leq 4 \), has been proved through the work of Colliot-Th\'el\`ene, Sansuc, and Swinnerton-Dyer \cite{CT-D-SD-1, CT-D-SD-2}. Singular cubic hypersurfaces with \( [K : k] = 3 \) and \( \deg(f(t)) \leq 3 \) have also been treated in work by Colliot-Th\'el\`ene and Salberger \cite{CT-S}.
Further results include situations where the extension \( K/k \) is arbitrary and \( f(t) \) has at most two roots in \( k \) (see \cite{CT-H-S, HB-Sko, Sch-Sko, Jones}), as well as the case of an irreducible quadratic \( f(t) \) defined over \( \mathbb{Q} \) with \( K/\mathbb{Q} \) arbitrary (see \cite{BHB, DSW}). V\'arilly-Alvarado and Viray \cite{VA-V} studied when $K/k$ is a cyclic extension of prime degree $p$ and $P(x)$ is a separable polynomial of degree $2p$. Moreover, assuming Schinzel's hypothesis, Colliot-Th\'el\`ene, Skorobogatov, and Swinnerton-Dyer \cite{CT-Sko-SD} have verified Colliot-Th\'el\`ene's conjecture for cyclic extensions \( K/k \) and general \( f(t) \).

In \cite[Theorem 1.6]{BST}, Browning, Sofos, and Ter\"av\"ainen proved that the Ch\^atelet varieties $X_{K, f}$ satisfy the integral Hasse principle for 100\% of polynomials $f \in \ZZ[t]$ of degree $d$, with positive leading coefficients. Since the Ch\^atelet varieties are not homogeneous, the above result does not guarantee the rational Hasse principle with probability $1$. For the rational Hasse principle, Skorobogatov and Sofos \cite[Theorem 1.3]{SS} have proved the Ch\^atelet varieties $X_{K, f}$ satisfy the rational Hasse principle for a positive proportion of polynomials $f \in \ZZ[t]$ of degree $d$. 

We study the rational points on the Ch\^atelet variety
	\begin{eqnarray*}
		\mathbf{N}_K(y_1, \dots, y_e) = \sum_{0 \leq i \leq d} c_i t^i.
	\end{eqnarray*}
	We write $t = u/v$ with $u, v \in \ZZ$ and $v \neq 0$, then the equation becomes
	\begin{eqnarray*}
		\mathbf{N}_K(y_1, \dots, y_e) v^d = \sum_{0 \leq i \leq d} c_i u^i v^{d-i}.
	\end{eqnarray*}
	Inspired by \cite{VA-V}, if we assume that $d$ is a multiple of $e = \deg(K / \QQ)$, say $d = be$, then we get 
	\begin{eqnarray}\label{norm-g}
		\mathbf{N}_K(x_1, \dots, x_e) = \sum_{0 \leq i \leq d} c_i u^i v^{d-i},
	\end{eqnarray}
where $x_i = y_i v^b$.
By homogeneity, $(\ref{norm-f})$ has a rational solution if and only if $(\ref{norm-g})$ has an integer solution with $v \neq 0$, assuming $e|d$.
	
Let $d \geq 2$ with $e | d$ and $\cc = (c_0, \dots, c_d) \in \ZZ^{d+1}$. Let $\uu = (u_1, u_2)$ and write
$$g_\cc(\uu) = c_0u_1^d + c_1 u_1^{d-1}u_2 + \dots + c_d u_2^d.$$
For fixed $K / \QQ$, we consider the variety $X_{K, g_\cc}$ defined by the equation
\begin{eqnarray}\label{cha-bino}
	\NNN_K(\xx) = g_{\cc}(\uu) \neq 0.
\end{eqnarray}
For a vector $\mathbf{a} = (a_1, \dots, a_n) \in \RR^n$, put
$$|\mathbf{a}| = \max(|a_1|, \dots, |a_n|).$$
Equation (\ref{cha-bino}) is said to be \emph{locally solvable} if it is solvable over $\QQ_p$ and $\RR$.
We now introduce a few definitions for the sets of vectors \(\cc \in \ZZ^{d+1}\) such that the solutions of (\ref{cha-bino}) have the corresponding properties. We define
\begin{align}\label{Def-S}
    \nonumber S(H) & = \{\cc \in \ZZ^{d+1}: |\cc| \leq H, \, c_0 \cdot c_d \neq 0 \}, \\
    S^{\mathrm{loc}}(H) & = \{\cc \in S(H): \text{(\ref{cha-bino}) is locally solvable} \}, \\
	\nonumber S^{\mathrm{glob}}(H) & = \{\cc \in S^{\mathrm{loc}}(H): \text{(\ref{cha-bino}) has a rational solution}\}.
\end{align}

\begin{theorem}\label{thm-HP}
	Let $d$ be a positive integer and $K$ be any finite extension of $\QQ$ of degree $e$ with $e|d$. Then we have
	$$\frac{\# S^{\mathrm{glob}}(H)}{\# S^{\mathrm{loc}}(H)} = 1 + O_{d,K,A} \big( (\log H)^{-A} \big),$$
	as $H \rightarrow \infty$.
\end{theorem}

In Proposition \ref{thm-cha}, we shall prove a stronger version of Theorem \ref{thm-HP}, in which we can show that for all but $O(H^{d+1} / \big(\log H)^A\big)$ choices of coefficient vectors $\cc \in S^{\mathrm{loc}}(H)$, the equation (\ref{cha-bino}) has at least $H^{\Delta}$ rational solutions, where $\Delta$ can be any positive real number less than $\frac{1}{de(e+3)}$.

%In Proposition \ref{thm-cha}, we shall prove a stronger version of Theorem \ref{thm-HP}, in which we can almost always produce at least $H^{\Delta}$ rational solutions to the equation $\NNN_K(\xx) = g_\cc(t)$, for an appropriate constant $\Delta > 0$ depending only on $d$ and $e$. This quatification will permit us to establish a weak form of Zariski density, for $100 \%$ of admissible polynomials.

%\begin{corollary}\label{weak-ZD}
%	Let $K/\QQ$ be a number field of degree $e \geq 2$ and let $D, L \geq 1$. Then, as admissible degree $d$ coefficient vectors are ordered by height, with probability 1, the rational solutions to the equation $\NNN_K(\xx) = g(u, v)$ are not contained in a union of $\leq L$ irreducible curves of degree $k \in [d(2e+7), D]$.
%\end{corollary}

\subsection*{Acknowledgments}

I am deeply grateful to my advisor Tim Browning for suggesting this problem and for the many valuable discussions that shaped this work. I would also like to thank Efthymios Sofos, Matteo Verzobio, and Shuntaro Yamagishi for discussions and insights that contributed to this paper. I am also very grateful to the anonymous referee for their careful reading and for the considerable effort they put into improving the manuscript.

% The strategy used to prove Theorem \ref{thm-HP} can be generalized to further show that $100\%$ of the coefficient vectors $\mathbf{c} \in S^{\mathrm{loc}}(H)$, the set of rational points $X_{K, g_\cc}(\QQ)$ on the variety defined by \eqref{cha-bino} is Zariski dense.  The key additional ingredient is a version of weak approximation.
%Let $$S^{\mathrm{ZD}}(H) = \{\cc \in S^\QQ(H): X_{K, g_\cc}(\QQ) \text{ is Zariski dense}\}.$$

%\begin{theorem}\label{thm-ZD}
%	Let $d$ be a positive integer and $K$ be any finite extension of $\QQ$ of degree $e$ with $e|d$. Then we have
%	$$\frac{\# S^{\mathrm{ZD}}(H)}{\# S^{\mathrm{loc}}(H)} = 1 + O_{d,K,A} \big( (\log H)^{-A} \big),$$
%	as $H \rightarrow \infty$.
%\end{theorem}

\section{Equidistribution controls sums for almost all binary forms}

This section extends the main technical tool from \cite[Theorem 2.2]{BST}, demonstrating that controlling averages of a general arithmetic function over values of random polynomials can be effectively achieved when the function is well-distributed in arithmetic progressions. We generalize this result from polynomials to binary forms.

For $x > 1$, we define $x_1 = \exp \big(\frac{\sqrt{\log x}}{\log \log x}\big)$.
For $B \in \NN$, let $\tau_B(n)$ denote the $B$-fold divisor function.

\begin{Theorem}\label{prop_general}
	Let $A \geq 1, \eps > 0$ and $0 \leq k < l \leq d$ be integers. Let $H \geq H_0(A, d)$,
 $$\exp \Big( (\log H)^{1/100} \Big) \leq x \leq H^{1/(2d)}.$$
 Let $F: \ZZ \rightarrow \CC$ be a function such that
	\begin{enumerate}
		\item $|F(n)| \leq \tau_B(|n|)$ for all $n \in \ZZ$ and for some $B \geq 1$ satisfying $200B^{2d+4} \leq A$;
		\item There exists a set of prime powers $\mathcal{Q} \subset [1, x^d] \cap \ZZ$ such that $$\sum_{q \in \mathcal{Q}} q^{-1/(4d)} \ll_A (\log x)^{-3(A+1)},$$ and such that for any $q \leq x^d$ that is not a multiple of any element of $\mathcal{Q}$, we have
		\begin{eqnarray*}
			\max_{\substack{1 \leq v \leq q \\ \gcd(v, q) \leq x_1}} \sup_{\substack{I \text{ interval} \\ |I| > H^{1 - \eps}x^{l-k} \\ I \subset [-2Hx^d, 2Hx^d]}} \frac{q}{|I|} \Big| \sum_{\substack{n \in I \\ n \equiv v \, (\mod q)}} F(n) \Big| \ll (\log H)^{-40Ad^2}.
		\end{eqnarray*}
	\end{enumerate}
	Then for any binary form 
 \begin{align}\label{g-form}
     g = \sum_{i = 0}^d c_i s^i t^{d-i} \in \ZZ[s,t]
 \end{align}
	of degree $d$ with $c_i \in [-H, H] \cap \ZZ$ such that $c_0 \cdot c_d \neq 0$ (that is, $(c_0,\dots,c_d) \in S(H)$), and for any coefficients $\alpha_{m,n} \in \CC$ such that $|\alpha_{m,n}| \leq 1$, we have 
\begin{align*}
		\sup_{\substack{x' \in [x/2, x]}} \, \sum_{|a|,|b| \leq H} \Big| \sum_{\substack{1 \leq m, n \leq x'}} \alpha_{m,n} F(am^kn^{d-k} + bm^ln^{d-l} + g(m, n)) \Big|^2 \ll H^2 x^4 (\log x)^{-A}.
\end{align*}
\end{Theorem}

\subsection{Preparatory results}

\begin{lemma}\label{Lemma-1.2}
    Let $q = p^e$ be a prime power and fix $a,b,c \in \NN$. We have
    \begin{align}\label{eq.1.2-main}
        \# \{ (v_1, v_2) \in (\NN \cap [1,x])^2: q | v_1^{a} v_2^{b} (v_1^c - v_2^c) \} \ll \frac{x^2}{q^{1 / (2 \max (a,b,c))}}.
    \end{align}
\end{lemma}
\begin{proof}
    Let $d = \gcd(v_1, v_2)$ and $v_1 = d u_1, v_2 = d u_2$. We have
    $$v_1^a v_2^b (v_1^c - v_2^c) = d^{a+b+c} u_1^a u_2^b (u_1^c - u_2^c).$$

    The left hand side of (\ref{eq.1.2-main}) is $\leq \sum_{k \geq 0} S_k$, where 
    \begin{align*}
      S_k =  \sum_{\substack{d \leq x \\ v_p(d) = k}}  \#\{ (u_1, u_2) \in (\NN \cap [1,x/d])^2: \gcd(u_1, u_2) = 1, p^{e} | p^{k(a+b+c)} u_1^{a} u_2^{b} (u_1^c - u_2^c) \}.
    \end{align*}
    When $k(a+b+c) \leq e$, we denote $e' = e - k(a+b+c)$. Then we know that
    \begin{align*}
        S_k & = \sum_{\substack{d \leq x \\ v_p(d) = k}}  \#\{ (u_1, u_2) \in (\NN \cap [1,x/d])^2: \gcd(u_1, u_2) = 1, p^{e'} | u_1^{a} u_2^{b} (u_1^c - u_2^c) \}.
    \end{align*}
    Since $\gcd(u_1,u_2) = 1$, the three integers $u_1^a, u_2^b$, and $u_1^c - u_2^c$ are pairwise coprime, and so $p^{e'}$ divides exactly one of them. For fixed $d$, there are at most $(x/d)^2 \cdot (p^{-e'/a} + p^{-e'/b})$ pairs of $(u_1, u_2)$ such that $p^{e'}$ divides $u_1^a$ or $u_2^b$. According to \cite[Lemma 2.3]{BST}, we have
    \begin{align*}
       \#\{ (u_1, u_2) \in (\NN \cap [1,x/d])^2: \gcd(u_1, u_2) = 1, p^{e} | (u_1^c - u_2^c) \} \ll \frac{(x/d)^2}{p^{e' / 2c}}.
    \end{align*}
    This yields
    \begin{align*}
        S_k \ll \sum_{\substack{d \leq x \\ v_p(d) = k}} \frac{(x/d)^2}{p^{e' / (2 \max (a,b,c))}} \ll \frac{(x / p^k)^2}{p^{ (e - k(a+b+c) ) / (2 \max (a,b,c))}}.
    \end{align*}
    When $k(a+b+c) > e$, we consider the trivial bound $S_k \leq (x / p^k)^2$. By combining these two ranges of $k$, the left hand side of (\ref{eq.1.2-main}) is 
    \begin{align*}
     & \ll \sum_{\substack{k \geq 0 \\ k(a+b+c) \leq e}} \frac{(x/p^k)^2}{p^{(e-k(a+b+c))/(2 \max (a,b,c))}} + \sum_{\substack{k \geq 0 \\ k(a+b+c) > e}} (x/p^k)^2 \\
     & \leq \sum_{\substack{k \geq 0 \\ k(a+b+c) \leq e}} \frac{x^2}{p^{e/(2 \max (a,b,c))}} +  \sum_{\substack{k \geq 0}} \frac{1}{p^{2k}} \cdot \frac{x^2}{p^{2e/(a+b+c)}} \\
     & \ll \frac{x^2}{q^{1/(2 \max (a,b,c))}} + \frac{x^2}{q^{2/(a+b+c)}} \\
     & \ll \frac{x^2}{q^{1/ (2\max(a,b,c))}}.\qedhere
    \end{align*}
\end{proof}

\begin{lemma}\label{Lemma-1.3}
	Let $0 \leq k < l \leq d$. Let $g$ be a binary form in the form (\ref{g-form}) of degree $d$ with $(c_0,\dots,c_d)\in S(H)$. Then for any $m, n \in \ZZ$, we have 
	\begin{eqnarray*}
		\gcd(g(m, n), m^k n^{d-l}) \leq \gcd(m,n)^d \gcd(c_0,m)^k \gcd(c_d,n)^{d-l}.
	\end{eqnarray*}
\end{lemma}
\begin{proof}
	Let $t = \gcd(m, n)$ and $m = m_0t, n = n_0t$. Then
	\begin{align*}
		\gcd(g(m,n), m^k n^{d-l}) & = t^{k+d-l} \gcd(t^{l-k}g(m_0,n_0), m_0^kn_0^{d-l}) \\
            & \leq t^{d} \gcd(g(m_0,n_0), m_0^k n_0^{d-l}) \\
		& = t^{d} \gcd(g(m_0,n_0), m_0^k) \gcd(g(m_0,n_0), n_0^{d-l}) \\
            & \leq t^{d} \gcd(g(m_0,n_0), m_0)^k \gcd(g(m_0,n_0), n_0)^{d-l} \\
		& \leq t^{d} \gcd(c_0, m_0)^k \gcd(c_d, n_0)^{d-l} \\
		& \leq t^d \gcd(c_0, m)^k \gcd(c_d, n)^{d-l}.\qedhere
	\end{align*}
\end{proof}

\begin{lemma}\label{Lemma-1.4}
   For all $q \in \NN$, $z \geq 1$ and $\alpha \geq 1/2$, we have
   $$\sum_{\substack{d | q \\ d \geq z}} d^{-\alpha} \ll z^{-\alpha+1/2} \exp \Big( \frac{4 (\log(3q))^{1/2}}{(\log \log (3q))^{3/2}} \Big).$$
\end{lemma}
\begin{proof}
   Writing $d^{-\alpha} = d^{-\alpha + 1/2} \cdot d^{-1/2} \leq z^{-\alpha + 1/2} d^{-1/2}$ for $d \geq z$, and bounding the divisor sum by the Euler product $\sum_{d \mid q} d^{-1/2} \leq \prod_{p \mid q}(1 - p^{-1/2})^{-1}$ via multiplicativity and the geometric series, we obtain the upper bound
   $$\sum_{\substack{d|q \\ d \geq z}} d^{-\alpha} \leq z^{-\alpha+1/2} \prod_{p | q} \Big( 1 - p^{-1/2} \Big)^{-1}.$$
   The remaining proof is the same as in the proof of \cite[Lemma 2.5]{BST}, where this Euler-product estimate is used.
\end{proof}

We will prove a generalization of \cite[Lemma 2.7]{BST}.
\begin{lemma}\label{Lemma-1.5}
	Let $A > 1, d \in \NN$ and $0 \leq k < l \leq d$ be fixed, and let $x \gg_A 1$. Let $g$ be a binary form in the form (\ref{g-form}) of degree $d$ with $(c_0,\dots,c_d)\in S(x^A)$. Define
\begin{equation}\label{eq:define_M}
\mathcal{M}_A=\left\{\mathbf (\mm, \nn) \in (\NN \cap [1,x])^4: 
\begin{array}{l}
m_i, n_i > x / (\log x)^{A}, i = 1,2 \\
\gcd(m_i, n_j) < (\log x)^{A}, 1 \leq i,j \leq 2 \\
|n_1m_2 - n_2m_1| > x^2/ (\log x)^{A+1} \\
\gcd(g(m_1,n_1), m_1^k n_1^{d-l}) < x_2
\end{array}
\right\},
\end{equation}
where $x_2 = \sqrt{x_1} = \exp(\sqrt{\log x}/ (2\log \log x ))$. Then 
\begin{eqnarray*}
	\#((\NN \cap [1,x])^4 \setminus \mathcal{M}_A) \ll x^4 / (\log x)^A.  
\end{eqnarray*}
\end{lemma}

\begin{proof}
	It is sufficient to show that the set of $(n_1, n_2)$ that fail any individual property in (\ref{eq:define_M}) has size less than the claimed bound. The proof for the first two properties is identical to that presented in \cite[Lemma 2.7]{BST}.

    For the third property, if we arbitrarily choose $n_1$ and $m_2$, the number of choices of $n_2$ and $m_1$ is at most
    \begin{align*}
        \sum_{|t - n_1m_2| \leq x^2/(\log x)^{A+1}} \tau(t) & \ll x^2 / (\log x)^A. 
    \end{align*}
    Hence the total possible choices are $\ll x^4 / (\log x)^A.$
    
    Finally, it suffices to estimate the size of pairs satisfying the second property that fail the last property. If $\gcd(g(m_1,n_1), m_1^k n_1^{d-l}) > x_2$, then it follows from Lemma \ref{Lemma-1.3} that 
    $$\gcd(m_1,n_1)^d \gcd(c_0, m_1)^k \gcd(c_d, n_1)^{d-l} > x_2.$$
    Since we assume that $\gcd(m_1, n_1) < (\log x)^A$, we have
    $$\gcd(c_0, m_1)^k \gcd(c_d, n_1)^{d-l} > x_2 / (\log x)^{Ad}.$$
    This implies that
    $$\gcd(c_0, m_1) \cdot \gcd(c_d, n_1) > x_2^{1/d} / (\log x)^{A}.$$
    Therefore, by Lemma \ref{Lemma-1.4} and the assumption $0 < |c_0|, |c_d| \leq x^A$, the number of tuples satisfying the second property but failing the last property is
    \begin{align*}
       & \leq x^2 \sum_{\substack{a | c_0, b | c_d \\ {x_2}^{1/d} / (\log x)^{A} < a b \leq x}} \ \sum_{\substack{m_1, n_1 \leq x \\ a | m_1, b | n_1}} 1 \\
       & \leq  x^4 \, \sum_{a | c_0} \frac{1}{a} \sum_{\substack{b | c_d \\ a b > {x_2}^{1/d}(\log x)^{-A}}} \frac{1}{b} \\
       & \ll x^4 \, \sum_{a | c_0} \frac{1}{a} \Big( \frac{{x_2}^{1/d}}{a (\log x)^A} \Big)^{-1/2} \exp \Big( \frac{4(\log (3 x^A))^{1/2}}{(\log \log (3 x^A))^{3/2}} \Big) \\
       & \ll x^4 \Big( \frac{{x_2}^{1/d}}{(\log x)^A} \Big)^{-1/2} \exp \Big( \frac{8(\log (3 x^A))^{1/2}}{(\log \log (3 x^A))^{3/2}} \Big) 
       \\
       & \ll x^4 (\log x)^{A/2} \exp \Big( \frac{8 \sqrt{A}(\log x)^{1/2}}{(\log \log x)^{3/2}}  - \frac{1}{4d} \frac{(\log x)^{1/2}}{\log \log x}\Big).\\
       & \ll x^4 (\log x)^{-A}.\qedhere
    \end{align*}
\end{proof}

We would like to recall the fact
\begin{align}\label{sub-tau}
    \tau_B(mn) \leq \tau_B(m) \tau_B(n),
\end{align}
and the standard upper bound
\begin{align}\label{sum-tau}
    \sum_{r \leq R} \tau_B(r)^k \ll R(\log R)^{B^k - 1},
\end{align}
for any $B \geq 1$ and $k \geq 1$.
We will also apply \cite[Lemma 2.8]{BST} quite often, which has shown that
\begin{align}\label{BST-2.8}
	\sum_{y \leq n \leq x + y} \tau_B(rn + a)^k \ll \tau_B(\gcd(a, r))^k x (\log x)^{B^k - 1} \log \log x,
\end{align}
for fixed $k \geq 1$, $A, B \geq 1, x \geq 3, y \leq x^A$ and $1 \leq a, r \leq x^A$.

\subsection{Proof of Theorem \ref{prop_general}}

Let $x' \in [x/2, x]$ and $g \in \ZZ[x,y]$ be chosen so that 
\begin{eqnarray*}
	 \sum_{|a|,|b| \leq H} \Big| \sum_{\substack{1 \leq m, n \leq x'}} \alpha_{m,n} F(am^kn^{d-k} + bm^ln^{d-l} + g(m, n)) \Big|^2
\end{eqnarray*}
is maximized. 
Opening up the square and recalling that $|\alpha_{m, n}| \leq 1$, it is sufficient to show that 
\begin{align}\label{claim 1}
	\sum_{\substack{1 \leq m_i, n_i \leq x' \\ i = 1, 2}} \alpha_{m_1, n_1} \overline{\alpha}_{m_2, n_2} \sum_{\substack{|a|, |b| \leq H}} F(h_{a,b}(m_1, n_1)) \overline{F}(h_{a,b}(m_2,n_2))  \ll H^2 x^4 (\log x)^{-A}.
\end{align}
Here
\begin{eqnarray*}
	h_{a,b}(m, n) = am^kn^{d-k} + bm^ln^{d-l} + g(m, n).
\end{eqnarray*}

Let $\mathcal{M} = \mathcal{M}_{3A}$, where $\mathcal{M}_{3A}$ is defined in (\ref{eq:define_M}). We denote
\begin{align}\label{Def-Delta}
    \Delta = (n_1m_2)^{l-k} - (n_2m_1)^{l-k}.
\end{align}
We consider separately the following two sets:
\begin{align*}
	\mathcal{N} & = \{ (\mm,\nn) \in ([1, x'] \cap \ZZ)^4: \exists \, q \in \mathcal{Q} \text{ s.t. } q | \Delta m_1^k n_1^{d-l} \} \cap \mathcal{M}, \\
	\mathcal{N}^c & = \{(\mm, \nn) \in ([1, x'] \cap \ZZ)^4\} \setminus \mathcal{N}.
\end{align*}

Let $v_1 = n_1m_2$ and $v_2 = n_2m_1$. By Lemma \ref{Lemma-1.2}, (\ref{sum-tau}), Cauchy--Schwarz and the assumption $\sum_{q \in \mathcal{Q}} q^{-1/(4d)} \ll (\log x)^{-3(A+1)}$, it follows that
\begin{align*}
   & \ \# \{ (\mm,\nn) \in (\NN \cap [1, x'])^4: \exists \, q \in \mathcal{Q} \text{ s.t. } q | \Delta m_1^k n_1^{d-l} \} \\
   \leq & \ \# \{ (\mm,\nn) \in (\NN \cap [1, x'])^4: \exists \, q \in \mathcal{Q} \text{ s.t. } q | \Delta v_2^k v_1^{d-l} \} \\
   \leq & \ \sum_{q \in \mathcal{Q}} \sum_{\substack{v_1, v_2 \leq x^2 \\ q | (v_1^{l-k} - v_2^{l-k}) v_2^k v_1^{d-l} }}  \tau(v_1) \tau(v_2)  \\
   \leq & \ \sum_{q \in \mathcal{Q}} \Big( \sum_{\substack{v_1, v_2 \leq x^2}} \tau(v_1)^2 \tau(v_2)^2 \Big)^{1/2}
   \Big(\sum_{\substack{v_1, v_2 \leq x^2 \\ q | (v_1^{l-k} - v_2^{l-k})v_2^k v_1^{d-l}}} 1 \Big)^{1/2} \\
   \ll & \ \sum_{q \in \mathcal{Q}} \frac{x^4 \log^3 x}{q^{1/{(4d)}}} \ll  x^4(\log x)^{-3A}.
\end{align*}
Combining with Lemma \ref{Lemma-1.5}, we may conclude that
\begin{align}\label{est-N^c}
    \# \mathcal{N}^c \ll x^4 (\log x)^{-3A}.
\end{align}

\subsubsection*{Contribution of $(\mm, \nn) \in \mathcal{N}^c$}

It follows from assumption (1) that the contribution to the left hand side of (\ref{claim 1}) from $(\mm,\nn)\in\mathcal{N}^c$ is
\begin{align*}
	\ll \sum_{(\mm, \nn) \in \mathcal{N}^c} \sum_{|a|, |b| \leq H} \tau_B(h_{a, b}(m_1, n_1)) \tau_B(h_{a, b}(m_2, n_2)).
\end{align*}
By Cauchy--Schwarz, this is 
\begin{align}\label{Nc}
	\ll (\# \mathcal{N}^c)^{1/2} \, \Big( \sum_{\substack{ |a|, |b| \leq H \\ 1 \leq m_1, m_2, n_1, n_2 \leq x}} \tau_B(h_{a, b}(m_1, n_1))^2 \tau_B(h_{a, b}(m_2, n_2))^2 \Big)^{1/2}.
\end{align}
From (\ref{sub-tau}), (\ref{sum-tau}), (\ref{BST-2.8}), and Cauchy--Schwarz, we know
\begin{align*}
	& \sum_{\substack{ |a|, |b| \leq H \\ 1 \leq m_1, m_2, n_1, n_2 \leq x}} \tau_B(h_{a, b}(m_1, n_1))^2 \tau_B(h_{a, b}(m_2, n_2))^2 \\
	= &  
 \sum_{\substack{ |a| \leq H \\ 1 \leq m_1, m_2, n_1, n_2 \leq x}} \sum_{|b| \leq H} \tau_B(h_{a, b}(m_1, n_1))^2 \tau_B(h_{a, b}(m_2, n_2))^2 \\
	\leq & \sum_{\substack{ |a| \leq H \\ 1 \leq m_1, m_2, n_1, n_2 \leq x}} \Big( \sum_{|b| \leq H} \tau_B(h_{a, b}(m_1, n_1))^4 \Big)^{1/2} \Big( \sum_{|b| \leq H} \tau_B(h_{a, b}(m_2, n_2))^4 \Big)^{1/2} \\
	\ll & \sum_{\substack{ |a| \leq H \\ 1 \leq m_1, m_2, n_1, n_2 \leq x}} \Big( H (\log H)^{B^4} \tau_B(m_1^l n_1^{d-l})^4  \Big)^{1/2} \Big( H (\log H)^{B^4} \tau_B(m_2^l n_2^{d-l})^4 \Big)^{1/2} \\
	\ll & \  H^2 (\log H)^{B^4} \sum_{\substack{ 1 \leq m_1, m_2, n_1, n_2 \leq x}} \tau_B(m_1^l n_1^{d-l})^2 \tau_B(m_2^l n_2^{d-l})^2 \\
	\ll & \ H^2 (\log H)^{B^4} \Big(\sum_{1 \leq m \leq x} \tau_B(m)^{2d} \Big)^4  \\
	\ll & \ H^2 (\log H)^{B^4} x^4 (\log x)^{4 B^{2d}} \\
	\ll & \ H^2 x^4 (\log x)^{100B^4 + 4B^{2d}}.
\end{align*}
Therefore, by (\ref{est-N^c}) we know that  (\ref{Nc}) is $\ll H x^4 (\log x)^{-A}$, since $50 B^4 + 2B^{2d} \leq 100B^{2d+4} \leq A/2$ by assumption.

\subsubsection*{Contribution of $(\mm,\nn) \in \mathcal{N}$}

By the triangle inequality and assumption (1), the left hand side of (\ref{claim 1}) is 
\begin{align}\label{claim 2}
     \ll \sum_{(\mm, \nn) \in \mathcal{N}} \sum_{u_2 \in \ZZ} \tau_B(u_2) \Big| \sum_{u_1 \in \ZZ} F(u_1) \gamma(\uu)  \Big|,
\end{align}
where
\begin{align*}
	\gamma(\uu) = \#\{ (a,b) \in (\ZZ \cap [-H,H])^2: u_i = h_{a,b}(m_i, n_i), i = 1,2 \}.
\end{align*}
We make the change of variables $u_i' = (u_i - g(m_i, n_i))/(m_i^{k}n_i^{d-l})$ to write this as
\begin{align*}
	\sum_{(\mm, \nn) \in \mathcal{N}} \sum_{u_2' \in \ZZ} \tau_B(u_2' m_2^k n_2^{d-l} + g(m_2,n_2)) \Big| \sum_{u_1' \in \ZZ} F(u_1'm_1^kn_1^{d-l} + g(m_1,n_1))  \gamma'(\uu')  \Big|,
\end{align*}
where
\begin{align}\label{eqn-uu'}
	\gamma'(\uu') = \#\{ (a,b) \in (\ZZ \cap [-H,H])^2: u_i' = a n_i^{l-k} + b m_i^{l-k}, i = 1,2 \}.
\end{align}

From (\ref{Def-Delta}), we have
$$|\Delta| \geq |n_1m_2 - n_2m_1| \cdot \min(n_1m_2, n_2m_1)^{l-k-1}.$$
When $(\mm,\nn) \in \mathcal{N}$, it follows (\ref{eq:define_M}) that
\begin{align}\label{est-Delta}
	\frac{x^{2(l-k)}}{(\log x)^{6A(l-k)}} \ll |\Delta| \ll x^{2(l-k)} .
\end{align}
The solution of (\ref{eqn-uu'}) is $a = \Delta^{-1}(m_2^{l-k}u_1' - m_1^{l-k}u_2')$ and $b = \Delta^{-1}(-n_2^{l-k}u_1' + n_1^{l-k}u_2').$ Hence we require
\begin{align}\label{eqn-u1'}
    m_2^{l-k}u_1' \equiv m_1^{l-k}u_2' \ (\mod \Delta), \quad n_2^{l-k}u_1' \equiv n_1^{l-k}u_2' \ (\mod \Delta).
\end{align}
This system has an integer solution if and only if
\begin{align}\label{cri-u1'}
    \gcd(m_2^{l-k}, \Delta) \, \big| \, m_1^{l-k}u_2', \quad \gcd(n_2^{l-k}, \Delta) \, \big| \, n_1^{l-k}u_2'.
\end{align}
In this case, the equation (\ref{eqn-u1'}) has a unique integer solution $u_1' \equiv u_1''$ modulo $\Delta_0$, where
\begin{align*}
    \Delta_0 = \lcm \Big( \frac{\Delta}{\gcd(m_2^{l-k}, \Delta)}, \frac{\Delta}{\gcd(n_2^{l-k}, \Delta)} \Big).
\end{align*}
Note that
\begin{align*}
    \gcd(n_2^{l-k}, \Delta) = \gcd(n_2^{l-k}, (m_2n_1)^{l-k}) \leq \big( \gcd(n_2, m_2) \gcd(n_2,n_1) \big)^{l-k} \leq (\log x)^{6A(l-k)},
\end{align*}
by (\ref{eq:define_M}). Hence by (\ref{est-Delta}) we have
\begin{align}\label{est-Delta0}
    \frac{x^{2(l-k)}}{(\log x)^{12A(l-k)}} \ll \Delta_0 \ll x^{2(l-k)} .
\end{align}

The condition $|a|, |b| \leq H$ is equivalent to
\begin{align*}
   |m_2^{l-k}u_1' - m_1^{l-k}u_2'| \leq |\Delta| H, \quad |-n_2^{l-k}u_1' + n_1^{l-k}u_2'| \leq |\Delta| H.
\end{align*}
For fixed $(\mm, \nn, u_2')$, we require that $u_1'$ lies in some interval $J(\mm,\nn, u_2')$ with
\begin{align}\label{est-J}
    |J(\mm,\nn, u_2')| \leq \frac{2 |\Delta| H}{\max (m_2, n_2)^{l-k}} \ll H x^{l-k}(\log x)^{3A(l-k)} .
\end{align}
By (\ref{eqn-uu'}), we also need $|u_2'| \leq (m_2^{l-k} + n_2^{l-k})H$.

We may now conclude that (\ref{claim 2}) is 
\begin{align}\label{want2bound}
    \leq \sum_{(\mm,\nn) \in \mathcal{N}} \sum_{|u_2'| \leq (m_2^{l-k} + n_2^{l-k}) H} \tau_B( u_2' m_2^k n_2^{d-l} + g(m_2,n_2)) | S_1 |,
\end{align}
where
\begin{align*}
    S_1 & = \sum_{\substack{u_1' \in J(\mm,\nn,u_2') \\ u_1' \equiv u_1'' \, \mod \Delta_0}} F\big(u_1' m_1^k n_1^{d-l} + g(m_1, n_1)\big),
\end{align*}
if (\ref{cri-u1'}) holds, and $S_1 = 0$ if it does not. Here $u_1' \equiv u_1''$ is the unique (if possible) solution for (\ref{eqn-u1'}) modulo $\Delta_0$.
We may rewrite
\begin{align*}
    S_1 = \sum_{\substack{u \in I \\ u \equiv v \, \mod q}} F(u),
\end{align*}
with modulus $q = \lcm(\Delta_0, m_1^k n_1^{d-l})$ and interval length
$|I| = m_1^k n_1^{d-l} J(\mm,\nn, u_2')$.
It follows that
\begin{align}\label{x_1=x_2^2}
    \gcd(v, q) & \leq \gcd(g(m_1,n_1), m_1^k n_1^{d-l}) \gcd(u_1'' m_1^k n_1^{d-l} + g(m_1,n_1), \Delta_0).
\end{align}
By definition (\ref{eq:define_M}), we see that
\begin{align}\label{2x_2}
    \gcd(g(m_1,n_1), m_1^k n_1^{d-l}) < x_2 = \exp(\sqrt{\log x}/2\log\log x)
\end{align}

We will now estimate (\ref{want2bound}) in three different subcases.

\subsubsection*{Subcase I}
 $\gcd(u_1'' m_1^k n_1^{d-l} + g(m_1,n_1), \Delta_0) \leq x_2$ and $|I| > H^{1 - \eps}x^{l-k}$. 

In this subcase, the assumption (2) of Theorem \ref{prop_general} is now satisfied, according to (\ref{x_1=x_2^2}) and (\ref{2x_2}). We have
$$|S_1| \ll \frac{|I|}{q (\log H)^{40 Ad^2}} = \frac{\gcd(\Delta_0, m_1^k n_1^{d-l}) J(\mm, \nn, u_2')}{\Delta_0 \, (\log H)^{40Ad^2}}.$$
It follows from (\ref{eq:define_M}) that
\begin{align*}
    \gcd(\Delta_0, m_1^k n_1^{d-l}) & \leq \gcd(\Delta, (m_1n_1)^d) \\
     & \leq \gcd((m_2n_2)^{d(l-k)}, (m_1n_1)^d) \\
    &  \leq \gcd \big(m_2n_2, m_1n_1\big)^{d^2} \\
    &  \ll (\log x)^{12Ad^2}.
\end{align*}
Hence (\ref{est-Delta0}) and (\ref{est-J}) yields that 
\begin{align*}
    |S_1| \ll \frac{H x^{l-k} (\log x)^{12Ad^2 + 15A(l-k)}}{x^{2(l-k)}  (\log H)^{40Ad^2}} \leq \frac{H}{x^{l-k} (\log H)^{12Ad^2}}.
\end{align*}
By (\ref{sub-tau}), (\ref{sum-tau}) and (\ref{BST-2.8}), we know (\ref{want2bound}) is 
\begin{align*}
	& \ll \sum_{m_1,m_2,n_1,n_2 \leq x} \ \sum_{|u_2'| \leq (m_2^{l-k} + n_2^{l-k}) H} \tau_B( u_2' m_2^k n_2^{d-l} + g(m_2,n_2)) \frac{H}{x^{l-k} (\log H)^{12Ad^2}} \\
	& \ll \sum_{m_1,m_2,n_1,n_2 \leq x} \tau_B(m_2^k n_2^{d-l}) H^2 (\log H)^{B - 12Ad^2} \\
        & \ll x^2 H^2 (\log H)^{B - 12 Ad^2} \sum_{m_2 \leq x} \tau_B(m_2)^k \sum_{n_2 \leq x} \tau_B(n_2)^{d-l} \\
        & \ll x^4 H^2 (\log H)^{B + B^k + B^{d-l} - 12 Ad^2}.
\end{align*}
As $B+B^k+B^{d-l} \leq 3B^d \leq 2A$ and $(\log H)^{-1} \leq (2d \log x)^{-1}$ by assumption, this shows that (\ref{want2bound}) is $\ll x^4H^2 (\log x)^{-A}$.

\subsubsection*{Subcase II}
 $\gcd(u_1'' m_1^k n_1^{d-l} + g(m_1,n_1), \Delta_0) \leq x_2$ and $|I| \leq H^{1 - \eps} x^{l-k}$.

We apply the crude bound $|F(n)| \leq \tau_B(|n|) \ll (Hx^d)^{\eps/10}$. This leads to the estimate
\begin{align*}
	|S_1| \ll (H x^d)^{\eps/10} \Big( \frac{|I|}{q} + 1 \Big).
\end{align*}
We use the assumption $|I| \leq H^{1 - \eps} x^{l-k}$, the fact $q \geq \Delta_0$ and (\ref{est-Delta0}) to deduce that
\begin{align*}
	 \frac{(H x^d)^{\eps/10} \, |I|}{q} \ll \frac{H^{1 - 0.7 \eps}x^{l-k}}{\Delta_0} \ll \frac{H}{x^{l-k}(\log H)^{12Ad^2}}.
\end{align*}
According to the assumption $x \leq H^{1/(2d)}$ in Theorem \ref{prop_general}, we get
\begin{align*}
	|S_1| \ll \frac{H}{x^{l-k}(\log H)^{12Ad^2}} + (H x^d)^{\eps/10} \ll \frac{H}{x^{l-k}(\log H)^{12Ad^2}}.
\end{align*}
The remaining part is similar to subcase I.

\subsubsection*{Subcase III}
 $\gcd(u_1'' m_1^k n_1^{d-l} + g(m_1,n_1), \Delta_0) > x_2$.

Let $u_1' = u' \Delta_0 + u_1''$. We rewrite
$$S_1 = \sum_{u' \in J'} F(u' \Delta_0 m_1^k n_1^{d-l} + u_1'' m_1^k n_1^{d-l} + g(m_1,n_1)),$$
where 
$$|J'| \ll |J| / \Delta_0 \ll H x^{k-l} 
(\log x)^{15A(l-k)},$$
by (\ref{est-Delta0}) and (\ref{est-J}).
Note that $(\ref{BST-2.8})$ along with the assumptions $|F(n)| \leq \tau_B(|n|)$ and $\log H \leq (\log x)^{100}$ allow us to bound 
\begin{align*}
    S_1  & \ll \tau_B(\Delta_0 m_1^k n_1^{d-l}) |J'| (\log |J'|)^B \\
    & \ll H x^{k-l} (\log x)^{15A(l-k) + 101B} \tau_B(\Delta_0 m_1^k n_1^{d-l}).
\end{align*}

For brevity, we denote
$$\sump_{u_2'}  = \sum_{\substack{|u_2'| \leq (m_2^{l-k} + n_2^{l-k}) H \\ \gcd(u_1'' m_1^k n_1^{d-l} + g(m_1,n_1), \Delta_0) > x_2}}.$$
We may bound (\ref{want2bound}) by
\begin{align*}
\ll H x^{k-l} (\log x)^{15A(l-k)+101B}  \sum_{(\mm,\nn) \in \mathcal{N}} \sump_{u_2'} \tau_B (u_2' m_2^k n_2^{d-l} + g(m_2,n_2)) \tau_B(\Delta_0 m_1^kn_1^{d-l}).
\end{align*}
By applying Cauchy--Schwarz, this is
\begin{align}\label{sc3-1}
    \leq H x^{k-l} (\log x)^{15A(l-k)+101B} S_2^{1/2} S_3^{1/2},
\end{align}
where
\begin{align*}
    S_2  = \sum_{(\mm,\nn)\in \mathcal{N}} \sum_{\substack{|u_2'| \leq (m_2^{l-k} + n_2^{l-k}) H}} \tau_B (u_2' m_2^k n_2^{d-l} + g(m_2,n_2))^2 \tau_B(\Delta_0 m_1^kn_1^{d-l})^2,
\end{align*}
and
\begin{align*}
    S_3 = \sum_{(\mm,\nn) \in \mathcal{N}} \sump_{u_2'} \, 1.
\end{align*}

It follows from (\ref{eqn-u1'}) that
\begin{align*}
	\gcd(u_1'' m_1^k n_1^{d-l} + g(m_1,n_1), \Delta_0)  & \leq \, \gcd(n_2^{l-k} u_1'' m_1^k n_1^{d-l} + n_2^{l-k} g(m_1,n_1), \Delta) \\
	& = \, \gcd(u_2' m_1^kn_1^{d-k} + n_2^{l-k} g(m_1,n_1), \Delta).
\end{align*}
Hence
$$S_3 \leq \sum_{(\mm,\nn) \in \mathcal{N}} \sum_{\substack{|u_2'| \leq (m_2^{l-k} + n_2^{l-k}) H \\ \gcd(u_2' m_1^kn_1^{d-k} + n_2^{l-k} g(m_1,n_1), \Delta) > x_2}} 1.$$
By (\ref{eq:define_M}), we know
\begin{align*}
	\gcd(m_1^kn_1^{d-k}, \Delta) & \leq \ \gcd(m_1^k, (m_2n_1)^{l-k}) \gcd(n_1^{d-k}, (m_1n_2)^{l-k}) \\
	& \ll \, \big( \gcd(m_1, m_2) \gcd(m_1, n_1) \gcd(n_1, m_1) \gcd(n_1, n_2) \big) ^{d^2}  \\
	& \ll \, (\log x)^{12Ad^2}.
\end{align*}
As a consequence of \cite[Lemma 2.6]{BST} and the fact $\tau(\Delta) \ll |\Delta|^{0.1/(l-k)} \ll x^{0.2}$, it follows that
\begin{align}\label{sc3-2}
    S_3 \ll & \ x^4 \left( \frac{x^{l-k} H (\log x)^{12Ad^2}}{x_2^{1/2}} \exp\Big( \frac{4(\log (3|\Delta|))^{1/2}}{(\log\log(3|\Delta|))^{3/2}} \Big) + \tau(\Delta) \right) \nonumber \\ 
     \ll & \ \frac{x^{4 +l-k} H (\log x)^{12Ad^2}}{x_2^{1/2}} \exp\Big( \frac{8 d^{1/2} (\log (3x))^{1/2}}{(\log\log x)^{3/2}} \Big) + x^{4.2}  \\ 
     \ll & \ \frac{x^{4+l-k} H}{\exp((\log x)^{0.49})}. \nonumber 
\end{align}

Now (\ref{sub-tau}), (\ref{BST-2.8}) and Cauchy--Schwarz imply that 
\begin{align*}
    S_2 & \ll x^{l-k} H (\log H)^B \sum_{(\mm,\nn) \in \mathcal{N}} \tau_B(\Delta)^2 \tau_B(m_1^k)^2 \tau_B(m_2^k)^2 \tau_B(n_1^{d-l})^2 \tau_B(n_2^{d-l})^2 \\
    & \ll x^{l-k} H (\log x)^{100B} \Big( \sum_{\substack{m_1,m_2,n_1,n_2 \leq x \\ m_1n_2 \neq m_2n_1}} \tau_B(\Delta)^4 \Big)^{1/2} \Big( \sum_{m \leq x} \tau_B(m)^{4d} \Big)^{2}. 
\end{align*}
Following from (\ref{sum-tau}), \cite[Lemma 2.9]{BST} and Cauchy--Schwarz, we obtain that
\begin{align*}
    \sum_{\substack{m_1,m_2,n_1,n_2 \leq x \\ m_1n_2 \neq m_2n_1}} \tau_B(\Delta)^4 & \ll \sum_{1 \leq t_1 < t_2 \leq x^2} \tau_B(t_1^{l-k} - t_2^{l-k})^4 \tau(t_1) \tau(t_2) \\
    & \ll \Big( \sum_{1 \leq t_1 < t_2 \leq x^2} \tau_B(t_1^{l-k} - t_2^{l-k})^8 \Big)^{1/2} \sum_{t \leq x^2} \tau(t)^2 \\
    & \ll x^4 (\log x)^{dB^8/2 + 3}
\end{align*}
Combine with (\ref{sum-tau}), we know
\begin{align*}
    S_2 \ll x^{4+l-k} H (\log x)^{100B + dB^8/4 + 3/2 + 2 B^{4d}}.
\end{align*}
Therefore, along with (\ref{sc3-1}) and (\ref{sc3-2}) the last subcase is crudely bounded by $\ll x^4 H^2 (\log x)^{-A}$.

\section{Chowla and Bateman--Horn for random binary forms}

Here we recall the key results by Browning, Sofos and Ter\"av\"ainen \cite[Proposition 3.1, Proposition 3.2]{BST} on the Liouville function and the von Mangoldt function in short intervals and progressions.

\begin{Prop}\label{prop_Liouville}
	Let $\eps > 0$ be small and fixed, and let $A \geq 1$ be fixed. There exists $x_0=x_0(A,\eps)$ such that the following holds for every $x \geq x_0$. Let $c_0 = 5/36$. Then there exists a set $\mathcal{Q} \subset [(\log x)^A, x^{c_0}]$ of integers satisfying
	\begin{equation}\label{smallQ-Liou}
		\sum_{q \in \mathcal{Q}} \frac{1}{q^{\eps}} \ll (\log x)^{-30A / \eps},
	\end{equation}
	such that for any integer $1 \leq q \leq x^{c_0}$ that is not a multiple of an element of $\mathcal{Q}$, we have 
	\begin{equation}
		\max_{\substack{1 \leq a \leq q \\ \gcd(a, q) \leq x^{\eps}}} \sup_{\substack{I \subset [1, x] \\ |I| \geq x^{1-c_0+2\eps} }} \Bigg| \sum_{\substack{n \in I \\ n \equiv a \, \mod q}} \lambda(n) \Bigg| \ll \frac{|I|}{q (\log x)^A}.
	\end{equation}
\end{Prop}

\begin{Prop}\label{prop_Lambda}
	Let $\eps > 0$ be small and fixed, and let $A \geq 1$ be fixed. There exists $x_0=x_0(A,\eps)$ such that the following holds for every $x \geq x_0$. Let $c_0 = 5/36$. Then there exists a set $\mathcal{Q} \subset [(\log x)^A, x]$ of integers satisfying (\ref{smallQ-Liou}), such that for any integer $1 \leq q \leq x^{c_0}$ that is not a multiple of an element of $\mathcal{Q}$, we have
	\begin{equation}
		\max_{\substack{1 \leq a \leq q \\ \gcd(a, q) \leq x^{\eps}}} \sup_{\substack{I \subset [1, x] \\ |I| \geq x^{1-c_0+2\eps} }} \Bigg| \sum_{\substack{n \in I \\ n \equiv a \, \mod q}} \Lambda(n) - \mathbf{1}_{\gcd(a,q) = 1} \frac{x}{\varphi(q)} \Bigg| \ll \frac{|I|}{q (\log x)^A}.
	\end{equation}
\end{Prop}
 
\subsection{Proof of Theorem \ref{thm-Chowla}}
We will apply Theorem \ref{prop_general} and Proposition \ref{prop_Liouville} to establish Theorem \ref{thm-Chowla}. We use Chebyshev's inequality in the form that a second moment bound controls the number of coefficient vectors for which the corresponding summand exceeds a prescribed threshold. Our goal is to prove that
\begin{eqnarray}\label{Liouville-1}
	\sum_{g \in \CCC} \Big| \sum_{u,v \leq x} \lambda (g(u,v)) \Big|^2 \ll x^4 H^2 (\log x)^{-A},
\end{eqnarray}
uniformly for $x \in [H^c, 2H^c]$, where $0<c<5/(31d)$ is as in Theorem \ref{thm-Chowla}. From this, Theorem \ref{thm-Chowla} will then follow directly from Chebyshev's inequality.

Suppose that $\CCC$ has the $k$th and $l$th coordinates as two of its free coordinates, with $k<l$. In order to prove (\ref{Liouville-1}), it suffices to show that
\begin{eqnarray}\label{g=0-1}
	\sum_{|a|,|b| \leq H} \Bigg| \sum_{m,n \leq x} \lambda(am^kn^{d-k} + bm^ln^{d-l} + g(m,n)) \Bigg|^2 \ll H^2 x^4 (\log x)^{-A},
\end{eqnarray}
uniformly for binary forms $g \in \ZZ[s,t]$ of degree $\leq d$ which have coefficients in $[-H,H]$ and satisfying zero $s$-coefficients at degree $k$ and $l$. We claim it suffices to prove that, for any coefficients $\alpha_{m,n} \in \CC$ such that $|\alpha_{m,n}| \leq 1$, we have
\begin{eqnarray}\label{g=0-2}
	\sum_{|a|,|b| \leq H} \Bigg| \sum_{m,n \leq x} \alpha_{m, n}\lambda(am^kn^{d-k} + bm^ln^{d-l} + g(m,n)) \Bigg|^2 \ll H^2 x^4 (\log x)^{-A},
\end{eqnarray}
uniformly for binary forms $g \in \ZZ[s,t]$ of degree $\leq d$ which have coefficients in $[-H,H]$ and satisfying the conditions $c_{d-k} = c_{d-l} = 0$ and $c_0 \cdot c_d \neq 0$, where
$$g(m, n) = c_0 m^d + c_1 m^{d-1}n + \cdots + c_{d-1}m n^{d-1} + c_d n^d.$$ 
Let $i \geq 0$ be the minimal integer such that $c_{d-i} \neq 0$, taking $i = d + 1$ if $g$ vanishes identically. Similarly let $j \leq d$ be the maximal integer such that $c_j \neq 0$, taking $j = -1$ if $g$ vanishes identically. Note that we have $d-i \geq j$.

If $i \leq k$ and $d-j \geq l$, we see that
\begin{align*}
\Sigma  = \sum_{|a|,|b| \leq H} \Bigg| \sum_{m,n \leq x} \alpha_{m,n} \lambda(am^{k-i}n^{d-k-j} + bm^{l-i}n^{d-l-j} + g^{*}(m,n)) \Bigg|^2,
\end{align*}
where $\alpha_{m,n} = \lambda(m)^i \lambda(n)^j$ and $g^{*}(m,n) =  c_{j}m^{d-i-j}  + \cdots + c_{d-i}n^{d-i-j}$ with $c_j \cdot c_{d-i} \neq 0$. This is exactly of the form (\ref{g=0-2}).

If $i > k$ and $d-j \geq l$, we have
\begin{align*}
\Sigma & = \sum_{|a|,|b| \leq H} \Bigg| \sum_{m,n \leq x} \alpha_{m,n} \lambda(an^{d-i-j} + bm^{l-i}n^{d-l-j} + g^{\dagger}(m,n)) \Bigg|^2 \\
& = \sum_{|a|,|b| \leq H} \Bigg| \sum_{m,n \leq x} \alpha_{m,n} \lambda\big((a+1)n^{d-i-j} + bm^{l-i}n^{d-l-j} + g^{\dagger}(m,n)\big) \Bigg|^2 + O(Hx^4),
\end{align*}
on shifting $a$ by $1$, where $\alpha_{m, n} = \lambda(m)^k \lambda(n)^j$ and 
$$g^{\dagger}(m, n) = c_j m^{d-i-k} + \dots + c_{d-i}m^{i-k}n^{d-i-j}.$$
The first term is exactly of the form (\ref{g=0-2}), with $g(m, n) = g^{\dagger}(m, n) + n^{d-i-j}$, and the error term of $O(Hx^4)$ is acceptable.
If $i \leq k$ and $d-j < l$, the argument follows similarly.

Finally, if $i > k$ and $d-j < l$, we have
\begin{align*}
\Sigma & = \sum_{|a|,|b| \leq H} \Bigg| \sum_{m,n \leq x} \alpha_{m,n} \lambda(an^{d-i-j} + bm^{d-i-j} + g^{\ddagger}(m,n)) \Bigg|^2 \\
& = \sum_{|a|,|b| \leq H} \Bigg| \sum_{m,n \leq x} \alpha_{m,n} \lambda( (a+1)n^{d-i-j} + (b+1)m^{d-i-j} + g^{\ddagger}(m,n)) \Bigg|^2+ O(Hx^4),
\end{align*}
on shifting both $a$ and $b$ by $1$, where $\alpha_{m, n} = \lambda(m)^k \lambda(n)^{d-l}$, and
$$g^{\ddagger}(m, n) = c_j m^{d-j-k} n^{j-d+l} + \dots + c_{d-i}m^{i-k}n^{l-i}.$$ 
The first term is exactly of the form (\ref{g=0-2}), with $g(m, n) = g^{\ddagger}(m,n) + n^{d-i-j} + m^{d-i-j}$, and the error term of $O(Hx^4)$ is acceptable.

\begin{proof}[Proof of Theorem \ref{thm-Chowla}]
By the discussion above, it suffices to prove (\ref{g=0-1}). We now proceed to verify the two assumptions of Theorem \ref{prop_general}. Assumption (1) is clear. Assumption (2) follows from Proposition \ref{prop_Liouville} with $\eps = 1/(8d)$, provided that $x^d \leq (2Hx^d)^{c_0}$ and $(2Hx^d)^{1-c_0} \leq H^{1-\eps}$, where $c_0 = 5/36$ as in Proposition \ref{prop_Liouville}. Given that $x \leq 2H^c$ with $c < 5/(31d)$, these conditions indeed hold for some small enough $\eps>0$. Therefore, (\ref{g=0-1}) follows by applying Theorem \ref{prop_general} with $B = 1$ to the function $F(n) = \lambda(n)$.
\end{proof}

\subsection{Proof of Theorem \ref{thm-BH}}

Let 
\begin{align}\label{w-and-W}
   w = \exp(\sqrt{\log x}), \indent W = \prod_{p \leq w} p.	
\end{align}
Define a modified Cram\'er model (or a $W$-tricked model) for the von Mangoldt function by
$$\Lambda_w(n) = \frac{W}{\varphi(W)} \mathbf{1}_{\gcd(n, W) = 1}.$$
Note that $\Lambda_w(n) \leq W / \varphi(W) \leq \log x$, for any $n \in \NN$. Recall the definition of $\mathfrak{S}_{g_1, \dots, g_r}(x)$ in Theorem \ref{thm-BH}, for an $r$-tuple $g_1, \dots, g_r \in \ZZ[s, t]$.

\begin{lemma}\label{mod-Cra}
	Let $A \geq 1$ and $d, r \in \NN$ be fixed, let $x \geq 2$, and let $w$ be as in (\ref{w-and-W}). Let $g_1, \dots, g_r \in \ZZ[s,t]$ be binary forms of degree $\leq d$. Then
	\begin{align}
		\frac{1}{x^2}\sum_{m, n \leq x} \Lambda_w(g_1(m,n)) \cdots \Lambda_w(g_r(m,n)) = \mathfrak{S}_{g_1, \dots, g_r}(x) + O(\log(x)^{-A}).
	\end{align}
\end{lemma}
\begin{proof}
	Let $$h(p) = \mathbf{1}_{p \leq w} \frac{\#\{(u, v) \in \FF_p^2: g_1(u, v) \cdots g_r(u, v) = 0 \}}{p^2}.$$  When $h(p) = 1$ for some prime $p$, Lemma \ref{mod-Cra} holds trivially. This is because in that case both $\Lambda_w(g_1(m,n)) \cdots \Lambda_w(g_r(m,n))$ and $\mathfrak{S}_{g_1, \dots, g_r}(x)$ would be zero. When $h(p) = 0$ for every prime $p \leq w$, $g_1 \cdots g_r$ does not have a fixed prime divisor $p \leq w$. In particular, for each $g_i$, $1 \leq i \leq r$, the content of $g_i$ (the greatest common divisor of its coefficients) is not divisible by $p$. By Lagrange's theorem on roots of a polynomial over a field, for any $1 \leq i \leq r$, we have
	\begin{align*}
		\#\{(u, v) \in \FF_p^2: g_i(u, v) = 0 \} & \leq p + \#\{(u, v) \in \FF_p^2: u \neq 0, g_i(u, v) = 0 \} \\
		& \leq p + (p-1)d  < p (d+1).
	\end{align*}
	We have
	$$h(p) \leq (d+1)rp/p^2 = (d+1)r/p,$$
	whenever $p \leq w$. It follows from Mertens' theorem that
	$$\prod_{a \leq p < b}(1 - h(p))^{-1} \leq \prod_{\substack{a \leq p < b \\ p \leq w}} \left( 1 - \frac{(d+1)r}{p} \right)^{-1} \leq \left( \frac{\log b}{\log a} \right)^{(d+1)r} \left( 1 + \frac{K_{d, r}}{\log a} \right), $$
	where $K_{d, r}$ is a constant depending on $d$ and $r$. Following the proof of \cite[Lemma 5.1]{BST}, we use the fundamental lemma of sieve theory \cite[Fundamental lemma 6.3]{Iwa-Kow} with $\kappa = (d+1)r$ to obtain the desired estimate.
\end{proof}

By Lemma \ref{mod-Cra}, the triangle inequality, Chebyshev's inequality and the induction argument in \cite[Proof of Theorem 1.4]{BST}, it suffices to prove for any coefficients $|\alpha_n| \leq 1$, we have
\begin{align}\label{BST-5.4}
	\sum_{g \in \CCC} \left| \sum_{m, n \leq x} \alpha_n(\Lambda - \Lambda_w)(g(m,n)) \right|^2 \ll \frac{x^2H^2}{(\log x)^A},
\end{align}
uniformly for $x \in [H^c, 2H^c]$. Suppose $\CCC$ has the $k$th and $l$th coordinates as two of its free coordinates, with $k < l$. It then suffices to prove that
\begin{align}\label{BST-5.5}
	\sum_{|a|, |b| \leq H} \left| \alpha_n(\Lambda - \Lambda_w) (a m^{k}n^{d-k} + bm^ln^{d-l} + g(m,n)) \right|^2 \ll \frac{x^2H^2}{(\log x)^A},
\end{align}
uniformly for sequences $|\alpha_n| \leq 1$, and for binary forms $g \in \ZZ[s, t]$ of degree $\leq d$ having coefficients in $[-H, H]$, and with $g(0,1) \cdot g(1, 0) \neq 0$. The condition $g(0,1) \cdot g(1, 0) \neq 0$ can indeed be imposed, since if $k = 0$ and $g(0,1) = 0$, or $l = d$ and $g(1, 0) = 0$, we can change the value of $g(0, 1)$ or $g(1, 0)$ by $O(1)$ without changing the validity of (\ref{BST-5.5}). Alternatively, if $k > 0$ or $l < d$, the $\CCC$ consists of binary forms which are not irreducible, a case that was excluded.

\begin{proof}[Proof of Theorem \ref{thm-BH}]
	As in the last paragraph of \cite[Proof of Theorem 1.4]{BST}, we can use \cite[Fundamental lemma 6.3]{Iwa-Kow} to conclude from Proposition \ref{prop_Lambda} that
$$\left| \sum_{\substack{x \leq n \leq x + h \\ n \equiv a \, (\mod q)}} (\Lambda(n) - \Lambda_w(n)) \right| \ll \frac{h}{q(\log x)^A},$$
provided that $x \geq h \geq x^{1-5/36+\eps}, 1 \leq a \leq q \leq x^{5/36}, \gcd(a, q) \leq \exp(\sqrt{\log x})$ and $q$ is not a multiple of any element of the set $\mathcal{Q}$ present in Proposition \ref{prop_Lambda}. Note that the function $F(n) = (\Lambda(n) - \Lambda_w(n))/(\log x)$ is $1$-bounded. Hence Assumption (1) and (2) of Theorem \ref{prop_general} both hold, provided that $(2Hx^d)^{1-5/36} \leq H^{1-\eps}$ for some small $\eps>0$, which indeed holds by the assumption that $x \leq 2H^c$ for some fixed $c < 5/(31d)$. Now (\ref{BST-5.4}) follows by applying Theorem \ref{prop_general} to $F(n)$.

\end{proof}

\section{Norm forms: a localized counting function}

Let $K/\QQ$ be a finite extension of degree $e \geq 2$ and fix a $\ZZ$-basis $\{\omega_1, \dots, \omega_e\}$ for the ring of integers of $K$. We denote the norm form by
\begin{align*}
	\mathbf{N}_K(\mathbf{x}) = N_{K/\QQ}(x_1 \omega_1 + \cdots + x_e \omega_e),
\end{align*}
where $\mathbf{x} = (x_1, \dots, x_e)$ and $N_{K/\QQ}$ is the field norm. Without loss of generality, we assume $\omega_1 = 1$, so we have 
\begin{align}\label{e-1}
	\NNN_K(\ee_1) = 1 \text{ for } \ee_1 = (1, 0, \dots, 0).
\end{align}

Recalling the definitions in (\ref{Def-S}), we further define $S^{\mathrm{loc}}_+(H)$ to be the subset of $S^{\mathrm{loc}}(H)$ such that there exists a real point $(\xx^{(\RR)}, \uu^{(\RR)} )$ such that
$$\NNN_K(\xx^{(\RR)}) = g_\cc(\uu^{(\RR)}) > 0.$$
We also denote $S^{\mathrm{glob}}_+(H)$ to be the subset of $S^{\mathrm{glob}}(H)$ such that there exists a rational point 
$(\xx^{(\QQ)}, \uu^{(\QQ)} )$ such that
$$\NNN_K(\xx^{(\QQ)}) = g_\cc(\uu^{(\QQ)}) > 0.$$
Similarly, we can define the corresponding subsets $S^{\mathrm{loc}}_-(H)$ and $S^{\mathrm{glob}}_-(H)$.

Let $\epsilon \in \{\pm\}$. If $\epsilon = +$, then we choose $\xx_1 = 2^{-1/e} \ee_1$, so that $\NNN_K(\xx_1) = 1/2$. If $\epsilon = -$ and $S^{\mathrm{loc}}_-(H)$ is non-empty, then there exists a real point $(\xx^{(\RR)}, \uu^{(\RR)} )$ such that
$\NNN_K(\xx^{(\RR)}) = g_\cc(\uu^{(\RR)}) < 0$. By homogeneity, there exists $\xx_1 \in \RR^{e}$ such that $\NNN_K(\xx_1) = -1/2$.
By Euler's homogeneous function theorem
\begin{align}\label{euler}
	e \, \NNN_K(\xx) = \nabla \NNN_K(\xx) \cdot \xx,
\end{align}
we know $\nabla\NNN_K(\xx_1) \neq 0$.
It follows that for $\epsilon \in \{\pm\}$, there exists $\kappa = \kappa(K, \epsilon) > 0$ such that, in the region 
\begin{equation}\label{def-reg-B}
	\mathcal{B} = \{ \mathbf{x} \in \RR^e: |\xx - \xx_1| < \kappa \},
\end{equation}
we have 
$$\frac{\pa \NNN_K}{\pa x_j}(\xx) \gg 1, \text{ for some } 1 \leq j \leq e.$$
%In particular $\mathbf{N}_K(\xx^{(\RR)}) = N_{K/Q}(\omega_1) = 1$.
For $B \geq 1$, we write $B \BB = \{B\xx : \xx \in \BB\}$.
Note that 
\begin{align}\label{bound-pa}
	\frac{\pa \NNN_K}{\pa x_j}(\xx) \gg B^{e-1}
\end{align}
throughout $B\BB$.
We shall be interested in the counting function $R_K(n; B)$ defined by
\begin{equation}\label{def-R}
R_K(n; B) = \# \{ \mathbf{x} \in \ZZ^e \cap B \mathcal{B}: \mathbf{N}_K(\mathbf{x}) = n \},
\end{equation}
for $n \in \ZZ$.

Let $x > 0$. Our principal concern is with the function 
$$N_{\cc}(x) = \sum_{\substack{-x \leq m, n \leq x \\ n \neq 0}} R_K(g_{\cc}(m,n); B),$$
as $x \rightarrow \infty$.
We shall always assume that
\begin{align}\label{x-H}
   \frac{1}{2}H^{\Delta_{d,e}} \leq  x \leq H^{\Delta_{d,e}}, \text{ with } \Delta_{d,e} = \frac{1}{2de(e+3)}.
\end{align}
We will take $B$ to be approximately $(Hx^d)^{1/e}$ and we will introduce the exact value of $B$ in (\ref{def-B}) later.

If $c_0 > 0$ then (\ref{cha-bino}) is soluble over $\RR$. From \cite[Theorem 1.4]{BBL}, we know
\begin{align}\label{Thm-BBL}
	\#S^{\mathrm{loc}}(H) \geq \left( \frac{1}{2} + o(1) \right) \prod_p \sigma_p \cdot \left( 2^{d+1} H^{d+1} + O(H^d) \right) \gg H^{d+1},
\end{align}
as $H \rightarrow \infty$, where $\sigma_p$ is the probability that the equation (\ref{cha-bino}) is solvable over $\QQ_p$, for a coefficient vector $\cc \in \ZZ_p^{d+1}$.
Hence it is sufficient to prove Theorem \ref{thm-HP} by showing the following stronger result.
\begin{prop}\label{thm-cha}
Let $d \geq 1$. Let $K/\QQ$ be a finite extension of degree $e$ dividing $d$. Suppose that conditions (\ref{e-1}) and (\ref{x-H}) hold. Then, for any $A > 0$ and any choice of $\epsilon \in \{\pm\}$, there exists a constant $C_A > 0$ such that
	$$\#\left\{\cc \in S_\epsilon^{\mathrm{loc}}(H): N_\cc(x) \leq \frac{x^2}{(\log H)^{C_A}} \right\} = O_{d,K,A} \left( \frac{ H^{d+1}}{(\log H)^{A}} \right).$$
\end{prop}

\subsection{A localized counting function}

Let 
\begin{align}\label{def-kx}
	k(x) = \lfloor 1000dA \log \log x \rfloor
\end{align}
and, for the remainder of the norm-form argument, redefine
\begin{align}\label{def-W}
	W = \prod_{p \leq w} p^{k(x)} \text{ \ for \ } w =\exp (\sqrt{\log x}).
\end{align}
This differs from the quantity denoted by $W$ in (\ref{w-and-W}): the present $W$ is the previous one raised to the power $k(x)$, and the notation follows \cite[\S 7]{BST}. In particular, by the prime number theorem, we have $\log W = (1 + o(1)) w k(x)$, and thus $W$ exceeds any power of $x$.

For any $a \in \ZZ$ and $q \in \NN$, we define
\begin{align*}
	\rho_K(q, a) = \frac{1}{q^{e-1}} \# \{ \sss \in (\ZZ / q\ZZ)^e: \NNN_K(\sss) \equiv a \, (\mod \, q)  \}.
\end{align*}
This function is multiplicative in $q$. With this notation, we put 
\begin{align}\label{def-Rh}
	\hat{R}_K(n; B) = \rho_K(W, n) \cdot \omega(n; B),
\end{align}
where $\omega(n; B)$ is an archimedean factor that we proceed to define.

By (\ref{bound-pa}), suppose now that $y = \NNN_K(\xx)$ with $\xx \in B\BB$. Then, by the implicit function theorem, we can express $x_j$ in terms of $\xx' = (x_1, \dots, \widehat{x_j}, \dots, x_e)$ and $y$ as $x_j = x_j(\xx', y)$. We now define
\begin{align}\label{def-archi-den}
	\omega(y) = \omega(y; B) = \int_{\BB(y)} \Big| \frac{\pa \NNN_K}{\pa x_j}(x_j(\xx', y), \xx') \Big|^{-1} d \xx',
\end{align}
where $\BB(y) = \{ \xx' \in \RR^{e-1}: (x_j(\xx', y), \xx') \in B \BB \}$. This expression represents the real density of solutions to the equation $y = \NNN_K(\xx)$. A similar real density was originally introduced in \cite[Lemma 9]{B-HB1}.

We recall \cite[Lemma 7.1]{BST}, which collects some basic facts about $\omega(y)$.

\begin{lemma}\label{lemma-arch}
	Let $\epsilon \in \{\pm\}$ and let $\mathcal{B}$ be the region defined by (\ref{def-reg-B}) which depends on $\epsilon$. The function $\omega(y) = \omega(y; B)$ in $(\ref{def-archi-den})$ is non-negative, continuously differentiable and satisfies the following properties:
	\begin{enumerate}[label=(\roman*)]
		\item we have
		$$\omega(y+h) - \omega(y) \ll B^{-e} |h|,$$
		for all $y, h \in \RR$;
		\item we have
		$$\int_I \omega(y) dy = \mathrm{vol} \{ \xx \in \BB: \NNN_K(\xx) \in I \}, $$
		for any interval $I \subset \RR$;
		\item $\omega$ is supported on an interval of length $O(B^e)$ centered on the origin and satisfies $\omega(y) \ll 1$ throughout its support;
		\item there exists an interval of length $\gg B^e$ around the point $\frac{\epsilon}{2} B^e$ on which we have $\omega(y) \gg 1$.
	\end{enumerate}
\end{lemma}

We now define the localized counting function to be
\begin{align}\label{def-lcf}
	\hat{N}_\cc(x) = \sum_{\substack{-x \leq m, n \leq x \\ n \neq 0}} \hat{R}_K(g_\cc(m, n); B),
\end{align}
for $\cc \in \ZZ^{d+1}$. The proof of Proposition \ref{thm-cha} proceeds in two steps. First we shall show that $\hat{N}_\cc(x)$ is usually a good approximation to $N_\cc(x)$, for suitable ranges of $x$ and $H$. The remaining task will then be to show that $\hat{N}_\cc(x)$ is rarely smaller than it should be.

\section{Norm forms: Approximation by the localized counting function}

For $\cc \in \ZZ^{d+1}$, the range of the form $g_\cc$ on the region $|u| = 1$ is a closed interval. We denote this interval by
\begin{align}
	I_\cc = \{ g_\cc(\uu): |\uu| = 1 \} = [b^-_\cc, b^+_\cc].
\end{align}
Note that if $\cc \in S_+^{\mathrm{loc}}(H)$, we require $b^+_\cc > 0$; if $\cc \in S_-^{\mathrm{loc}}(H)$, we require $b^-_\cc < 0$. 

We begin with a simple lemma that controls the difference in the maximum (or minimum) of two continuous functions when they differ by a bounded range at each point.
\begin{lemma}\label{min-max}
	Let $n \geq 1$ and let $X$ be a compact subset in $\RR^n$. Suppose that $f, g: X \rightarrow \RR$ are continuous functions satisfying
	$$b^- \leq f(\xx) - g(\xx) \leq b^+, \text{ for all } \xx \in X.$$
	Then
	$$b^- \leq \max_{\xx \in X} f(\xx) - \max_{\xx \in X}g(\xx) \leq b^+ \quad \text{ and } \quad b^- \leq \min_{\xx \in X} f(\xx) - \min_{\xx \in X}g(\xx) \leq b^+.$$
\end{lemma}
\begin{proof}
	For $\yy \in X$, we have
	\begin{align*}
		\max_{\xx \in X} f(\xx) \geq f(\yy) \geq g(\yy) + b^- \quad \text{ and } \quad \max_{\xx \in X} g(\xx) \geq g(\yy) \geq f(\yy) - b^+.
	\end{align*}
	Taking the maximum over $\yy \in X$ gives $$b^- \leq \max_{\xx \in X} f(\xx) - \max_{\xx \in X}g(\xx) \leq b^+.$$
	An analogous argument establishes the bounds for the minima.
\end{proof}

The next lemma shows that only a very small number of coefficient vectors $\cc \in S(H)$ result in the form $g_\cc$ attaining a small maximum (or minimum), in the absolute sense, on the region $|u| = 1$.
\begin{lemma}\label{Lemma-5.2}
    Let
\begin{align}\label{def-hat-A}
	\hat{H} = H(\log H)^{-A}.
\end{align}
    We have
    \begin{enumerate}
        \item $$\#\{\cc \in S(H): 0 < b^+_\cc \leq \hat{H} \} \ll \hat{H}H^{d}.$$
        \item $$\#\{\cc \in S(H): -\hat{H} \leq b^-_\cc < 0 \} \ll \hat{H}H^{d}.$$
    \end{enumerate}
\end{lemma}
\begin{proof}
We only prove (1); the proof of (2) is similar.

When $d$ is odd, we have $g_\cc(\uu) = -g_{\cc}(-\uu)$. By considering the four points $(\pm 1, 0)$ and $(0, \pm 1)$, we know
$$b^+_\cc \geq \max(\pm c_0, \pm c_d) = \max(|c_0|, |c_d|).$$ 
It follows that $\max(|c_0|, |c_d|) \leq \hat{H}$, if $b^+_\cc \leq \hat{H}$. Therefore, we get
$$\#\{\cc \in S(H): 0 < b^+_\cc \leq \hat{H} \} \leq \#\{\cc \in S(H): b^+_\cc \leq \hat{H} \} \ll \hat{H}^2 H^{d-1} \leq \hat{H}H^{d}.$$

Now we assume that $d$ is even.
	Let 
	$$b^{(+,1)}_{\cc} = \max\{g_\cc(\uu): |u| = 1, |v| \leq 1\}, \indent b^{(+,2)}_{\cc} = \max\{g_\cc(\uu): |u| \leq 1, |v| = 1\},$$
	so that
	\begin{align}\label{prop-bc}
		b^+_{\cc} = \max(b^{(+,1)}_{\cc}, b^{(+,2)}_{\cc}).
	\end{align}	
	Since $d$ is even, writing
	$g_\cc(\uu) = u^d + g_{(c_0-1, c_1, \dots, c_d)}(\uu)$ shows that
	\begin{align*}
		g_\cc(\uu) - g_{(c_0-1, c_1, \dots, c_d)}(\uu) = 1, & \text{ for any } \uu \text{ with } |u| = 1, \\
	 0 \leq g_\cc(\uu) - g_{(c_0-1, c_1, \dots, c_d)}(\uu) \leq 1, & \text{ for any } \uu \text{ with } |u| \leq 1.
	\end{align*}
	Similarly, we have
	\begin{align*}
		g_\cc(\uu) - g_{(c_0, c_1, \dots, c_d-1)}(\uu) = 1, & \text{ for any } \uu \text{ with } |v| = 1, \\
	 0 \leq g_\cc(\uu) - g_{(c_0, c_1, \dots, c_d-1)}(\uu) \leq 1, & \text{ for any } \uu \text{ with } |v| \leq 1.
	\end{align*}
	Thus, by Lemma \ref{min-max} we know
	\begin{align*}
		b^{(+,1)}_{\cc} - b^{(+,1)}_{(c_0-1, c_1, \dots, c_d)} = 1, \indent 0 \leq b^{(+,1)}_{\cc} - b^{(+,1)}_{(c_0, c_1, \dots, c_d-1)} \leq 1,\\
		b^{(+,2)}_{\cc} - b^{(+,2)}_{(c_0, c_1, \dots, c_d-1)} = 1, \indent 0 \leq b^{(+,2)}_{\cc} - b^{(+,2)}_{(c_0-1, c_1, \dots, c_d)} \leq 1,
	\end{align*}
	which imply that
	\begin{align*}
		1 \leq b^{(+,1)}_{\cc} - b^{(+,1)}_{(c_0-1, c_1, \dots, c_d-1)} \leq 2, \indent 1 \leq  b^{(+,2)}_{\cc} - b^{(+,2)}_{(c_0-1, c_1, \dots, c_d-1)} \leq 2.
	\end{align*}
	Therefore, by Lemma \ref{min-max} and (\ref{prop-bc}) we have
	\begin{align*}
		1 \leq b^+_\cc - b^+_{(c_0-1, c_1 \dots, c_{d-1}, c_d-1)} \leq 2.
	\end{align*}
	If $0 < b^+_\cc \leq \hat{H}$, then the number of possibilities for the choices of the pair $(c_0, c_d)$ with fixed $(c_1, \dots, c_{d-1})$ and the difference $c_d - c_0$ is $\ll \hat{H}$.
	Since there are $\ll H^{d}$ choices for $(c_1, \dots, c_{d-1})$ and $c_d - c_0$, we deduce that 
\begin{equation*}
		\#\{\cc \in S(H): 0 < b_\cc^{+} \leq \hat{H} \} \ll \hat{H}H^d.\qedhere
\end{equation*}
\end{proof}

It will be convenient to study a refinement in which the extreme value $b^-_\cc$ or $b^+_\cc$ is restricted to lie in a dyadic interval. When $H > 0$ is sufficiently large, let $\tilde{H}$ be a parameter satisfying
\begin{align}\label{def-tilde-H}
	2H(\log H)^{-A} \leq \tilde{H} \leq H.
\end{align}
We define
\begin{align}\label{def-StH}
	S_+(H, \tilde{H}) = \left\{ \cc \in S(H): \frac{1}{2}\tilde{H} < b^+_\cc \leq \tilde{H} \right\},
\end{align}
\begin{align}\label{def-StH-}
	S_-(H, \tilde{H}) = \left\{ \cc \in S(H): -\tilde{H} \leq b^-_\cc < -\frac{1}{2}\tilde{H} \right\}.
\end{align}
For $\epsilon \in \{\pm\}$, we further define
\begin{align*}
	S_\epsilon^{\mathrm{loc}}(H, \tilde{H}) = S_\epsilon(H, \tilde{H}) \cap S^{\mathrm{loc}}(H).
\end{align*}
We also set
\begin{align}\label{def-B}
	B = \tilde{H}^{1/e} x^{d/e}.
\end{align}

We state \cite[Lemma 8.2]{BST}, which characterizes important properties of the function 
$$F(n) = R_K(n; B) - \hat{R}_K(n; B).$$
\begin{lemma}\label{BST-Lemma-8.2}
Let $\varepsilon > 0$ be fixed. Let $B \geq 1$. Define
\begin{equation}\label{calQ}
\mathcal{Q} = \left\{ p^{k(x)} : p \leq w \right\}.
\end{equation}

Let $I \subset \mathbb{R}$ be an interval such that $I \subset [-2Hx^d, 2Hx^d]$, $|I| > H^{1 - \varepsilon}$ and $x \leq H^{\Delta_{d, e}}$ with $\Delta_{d, e}$. Let $q \in \mathbb{N}$ and let $1 \leq u \leq q$. \textit{Assume that $q$ is not divisible by any element of $\mathcal{Q}$, that $q \leq x^d$ and that}
\begin{equation*}
\gcd(u, q) \leq \exp\left( \frac{2 \sqrt{\log x}}{\log \log x} \right).
\end{equation*}
Then
\begin{align}\label{eqn-8.1}
	\left| \sum_{\substack{n \in I \\ n \equiv u \bmod q}} F(n) \right| 
\ll \frac{|I|}{q} \left( \exp\left( -\frac{1}{23} \sqrt{\log x} \right) + \frac{x^{d(e+2)} B^{e - 1/2}}{|I|} \right).
\end{align}

\end{lemma}
Moreover, from \cite[Section 8.1]{BST}, we know $|F(n)| \ll \tau(n)^{e+1}$ and the size of the exceptional set $\mathcal{Q}$ is $O((\log x)^{-100A})$.
In particular, these properties of \( F(n) \) together with (\ref{eqn-8.1}) satisfy the assumptions of Theorem \ref{prop_general}.
 
We now state our main result in this section, which shows that the localized counting function provides a good approximation to the global counting function for almost all coefficient vectors $\cc$.
\begin{prop}\label{cha-2moment}
	Let $C \geq 1$ be fixed. Assume conditions (\ref{x-H}) and (\ref{def-tilde-H}) hold. For $\epsilon \in \{ \pm\}$ and $\delta > 0$, define
	$$ E_{\epsilon, \delta}(x, H) = \# \left\{ \cc \in S_\epsilon(H, \tilde{H}): |N_\cc(x) - \hat{N}_\cc(x)| > \delta x^2 \right\},$$
	where $S_\epsilon(H, \tilde{H})$ is given by (\ref{def-StH}) and (\ref{def-StH-}). 
	Then
	$$E_{\epsilon, \delta}(x, H) \ll \frac{H^{d+1}}{\delta^2(\log x)^{2C}}. $$
\end{prop}
\begin{proof}
	We set the coefficients $\alpha_{m, n} = 1$ for all $n \neq 0$ and $\alpha_{m, n} = 0$ for $n = 0$. Then by applying Chebyshev's inequality together with Theorem \ref{prop_general}, we obtain
	\begin{align*}
		E_{\epsilon, \delta}(x, H) & < \frac{1}{\delta^2 x^4} \sum_{\cc \in S_\epsilon(H, \tilde{H})} |N_\cc(x) - \hat{N}_\cc(x)|^2 \\
		& = \frac{1}{\delta^2 x^4} \sum_{\cc \in S_\epsilon(H, \tilde{H})} \left| \sum_{\substack{-x \leq m,n \leq x \\ n \neq 0}} F(g_\cc(m,n)) \right|^2  \\ 
		& \ll \frac{1}{\delta^2 x^4} \cdot H^{d+1} x^4 (\log x)^{-2C} = \frac{H^{d+1}}{\delta^2(\log x)^{2C}}.\qedhere
	\end{align*}
\end{proof}

%\subsection{Proof of Corollary \ref{weak-ZD}}

%Any $(\xx, u, v) \in \ZZ^{e+2}$ counted by $N_\cc(x)$ is lie in the box $[-P, P]^{e+2}$ for 
%$$P \ll B \leq H^{1/e} x^{d/e} \leq H^{1/e + d \Delta_{d, e} / e}.$$
%Let $\eps > 0$ be fixed. According to Proposition \ref{thm-cha}, with probability 1, as admissible coefficient vectors are ordered by height, there are $\gg H^{2\Delta_{d,e} - \eps}$ integer solutions to the equation $\NNN_K(\xx) = g(u, v)$. We now appeal to a result of Bombieri and Pila \cite{XXX}, which shows that an irreducible degree $k$ curve in $\AAA^{e+2}$ has $O_{e, k, \eps}(P^{1/k + \eps})$ points in $(\ZZ \cap [-P, P])^{e+1}$, where the implied constant depends only on $e, k$ and $\eps$. Thus, if all the solutions were covered by $\leq L$ irreducible curves of degree $k \leq D$, we would have
%$$H^{2\Delta_{d,e} - \eps} \ll_{d,D,L,\eps} P^{1/k + \eps} \ll H^{(1/e + d \Delta_{d,e} / e)/k + \eps},$$
%which is a contradiction for 
%$$k > $$

\section{Norm forms: the localized counting function is rarely small}

Our goal is to show that for almost all coefficient vectors $\cc$, the localized counting function $\hat{N}_\cc(x)$ is large. More precisely, we prove the following result.

\begin{prop}\label{lcf-rs}
	Let $A \geq 1$ be fixed as in (\ref{def-tilde-H}). Take $B = \tilde{H}^{1/e}x^{d/e}$ and assume that conditions (\ref{x-H}) and (\ref{def-tilde-H}) hold. Let $\epsilon \in \{\pm\}$. Then there exists $C = C_{A, d ,e} \geq 1$, such that
	\[
\# \left\{ \mathbf{c} \in S_\epsilon^{\text{loc}}(H, \tilde{H}) : \hat{N}_\mathbf{c}(x) < \delta x^2 \right\} 
\ll \frac{H^{d+1}}{(\log x)^{dC}} 
+ \left(\delta (\log x)^{c_{d,e}} (\log H)^{2A} \right)^{\frac{1}{d^2 (e+1)}} \tilde{H} H^d,
\]
where $\delta = (\log x)^{-C}$ and $c_{d,e} = e + d^2(d+1)^{e+2}$.
\end{prop}

\begin{proof}[Proof of Proposition \ref{thm-cha} assuming Proposition \ref{lcf-rs}]
Let $C \geq 1$ and $\delta = (\log x)^{-C}$. We observe that
  \begin{align*}
  	\#\left\{\cc \in S_\epsilon^{\mathrm{loc}}(H): N_\cc(x) \leq \frac{x^2}{(\log H)^{C}} \right\} \leq \sum_{\substack{1 \leq \tilde{H} \leq H \\ \tilde{H} = 2^\alpha}} E(H, \tilde{H}),
  \end{align*}
where
  \begin{align*}
  	E(H, \tilde{H}) & = \#\left\{\cc \in S_\epsilon^{\mathrm{loc}}(H, \tilde{H}): N_\cc(x) \leq \frac{x^2}{(\log H)^{C}} \right\} \\
  	& \leq \#\left\{\cc \in S_\epsilon^{\mathrm{loc}}(H, \tilde{H}): |N_\cc(x) - \hat{N}_\cc(x)| > \delta x^2 \right\}  \\
  	& \quad \quad \quad \quad \quad  + \#\left\{\cc \in S_\epsilon^{\mathrm{loc}}(H, \tilde{H}): \hat{N}_\cc(x) < 2 \delta x^2 \right\}.
  \end{align*}
  Using Propositions \ref{cha-2moment} and \ref{lcf-rs} together with (\ref{x-H}), thus we deduce that
  $$E(H, \tilde{H}) \ll \frac{H^{d+1}}{(\log H)^{A+1}},$$ 
  provided condition (\ref{def-tilde-H}) holds and $C$ is chosen to be sufficiently large in terms of $A, d$ and $e$. Recall the definition (\ref{def-hat-A}) of $\hat{H}$. By applying Lemma \ref{Lemma-5.2}, we have
	\begin{align*}
		\sum_{\substack{1 \leq \tilde{H} \leq 2H \\ \tilde{H} = 2^j}} E(H, \tilde{H}) & \leq \sum_{\substack{\tilde{H} \leq \hat{H} \\ \tilde{H} = 2^j}} \tilde{H} H^d + \sum_{\substack{\hat{H} \leq \tilde{H} \leq H \\ \tilde{H} = 2^j}} E(H, \tilde{H}) \\
		& \ll \hat{H} H^d + (\log H) \frac{H^{d+1}}{(\log H)^{A+1}} \\
		& \ll \frac{H^{d+1}}{(\log H)^A}.\qedhere
	\end{align*}
\end{proof}

\subsection{The archimedean densities are rarely small}
First, we turn our attention to the archimedean aspect. 

For $\uu = (u, v) \in \RR^2 \setminus \{0\}$, let $(r, \theta) = (r(\uu), \theta(\uu))$ be its polar coordinates so that
$$u = r \cos \theta, \quad v = r \sin \theta, \quad \text{ with } r > 0 \text{ \ and \ }  0 \leq \theta < 2\pi.$$ 
For $0 \leq \theta < 2\pi$, we also define
$$g_\cc(\theta) = g_\cc(\cos \theta, \sin \theta).$$
Since $g_\cc(\uu)$ is a binary form of degree $d$, we have
$$g_\cc(\uu) = r^d g_\cc(\theta).$$

For $\epsilon \in \{\pm\}$ and $\cc \in S_\epsilon(H, \tilde{H})$, we choose $\uu_\cc = (u_\cc, v_\cc)$ with $|\uu_\cc| = 1$ satisfying $g_\cc(\uu_\cc) = b^\epsilon_\cc$. Thus, by the definitions (\ref{def-StH}) and (\ref{def-StH-}), we have
$$\frac{1}{2}\tilde{H} < |g_\cc(\uu_\cc)| \leq \tilde{H}.$$ 
Let $\theta_\cc = \theta(\uu_\cc)$, then $g_\cc(\theta_\cc) = r(\uu_\cc)^{-d} g_\cc(\uu_\cc)$. Since $1 \leq r(\uu_\cc) \leq \sqrt{2}$, we know
$$(\sqrt{2})^{-d} |g_\cc(\uu_\cc)| \leq |g_\cc(\theta_\cc)| \leq |g_\cc(\uu_\cc)|.$$
This implies that
\begin{align}\label{theta-H}
	(\sqrt{2})^{-d-2} \tilde{H} \leq |g_\cc(\theta_\cc)| \leq \tilde{H}.
\end{align}

\begin{lemma}\label{lbd-arc}
Let $\epsilon \in \{\pm\}$ and $\cc \in S_\epsilon(H, \tilde{H})$.
    Then there exists a small constant $0 < \eta \ll 1$ such that if
    $$\Big|r/x - \Big(\frac{1}{2}\tilde{H}/|g_\cc(\theta_\cc)|\Big)^{1/d} \Big| < \eta^2 \ \text{ and } \ \norm{\theta - \theta_\cc} \leq (\log H)^{-2A},$$
    then
    $$\omega(r^d g_\cc(\theta); B) \gg 1.$$
    Here $\norm{\theta - \theta_\cc} = \min \{|\theta - \theta_\cc + 2\pi n|: n \in \ZZ \}.$
\end{lemma}
\begin{proof}
    It follows from Lemma \ref{lemma-arch} (iv) that there exists $0 < \eta \ll 1$ such that $\omega(y; B) \gg 1$ whenever 
    $$|\epsilon y/B^e - 1/2| < \eta.$$ 
    If $\big|r/x - \big(\frac{1}{2}\tilde{H}/|g_\cc(\theta_\cc)|\big)^{1/d} \big| < \eta^2$, then 
    \begin{align}\label{6.1.1}
    	(r/x)^d = \frac{\tilde{H}}{2|g_\cc(\theta_\cc)|} + O(\eta^2),
    \end{align}
    noting from (\ref{theta-H}) that we have $g_\cc(\theta_\cc) \gg_d \tilde{H}$.
	Moreover, Taylor's theorem yields
	$$\Big||g_\cc(\theta)| - |g_\cc(\theta_\cc)|\Big| \leq |g_\cc(\theta) - g_\cc(\theta_\cc)| \ll \sum_{k \geq 1} \norm{\theta - \theta_\cc}^k \cdot \Big| \frac{\pa^k g_\cc}{\pa \theta^k} (\theta_\cc) \Big| \ll H(\log H)^{-2A}. $$
	Since $r \leq \sqrt{2}x$, it follows that
	\begin{align}\label{6.1.2}
		r^d g_\cc(\theta) = r^d g_\cc(\theta_\cc) + O(x^dH(\log H)^{-2A}). 
	\end{align}
	Recall from (\ref{def-B}) that $B^e = \tilde{H}x^d$. It follows from (\ref{theta-H}) that 
	$$\frac{|g_\cc(\theta_\cc)|}{\tilde{H}} \leq 1.$$ By (\ref{6.1.1}) and (\ref{6.1.2}), we obtain 
	\begin{align*}
	    \epsilon r^d g_\cc(\theta) / B^e - 1/2 &  = r^d |g_\cc(\theta)| / B^e - 1/2  \\
		 &  = \frac{2r^d |g_\cc(\theta_\cc)| - \tilde{H}x^d + O(x^dH(\log H)^{-2A})}{2\tilde{H}x^d} \\
		 &  = O \left( \frac{\eta^2|g_\cc(\theta_\cc)|}{\tilde{H}} + \frac{H}{\tilde{H}(\log H)^{2A}} \right) \\
		 & = O\left(\eta^2 + \frac{H}{\tilde{H}(\log H)^{2A}} \right),
	\end{align*}
	The right-hand side is $\ll \eta$, since $\tilde{H} \geq 2H(\log H)^{-A}$. The statement of the lemma follows by choosing $\eta > 0$ small enough.
\end{proof}

For $\epsilon \in \{\pm\}$ and $\cc \in S_\epsilon(H, \tilde{H})$, define 
$$r_0 = \left(\frac{1}{2}\tilde{H}/|g_\cc(\theta_\cc)|\right)^{1/d} x.$$
From (\ref{theta-H}), we know $2^{-1/d} x \leq r_0 \leq \sqrt{2} x$.
We now introduce the region 
\begin{equation}\label{def-Dc}
D_\cc = \left\{(u, v) \in [0, x]^2 : 
\begin{array}{l}
|r - r_0| \leq \eta^2 x, \\
\|\theta - \theta_\cc\| \leq 3(\log H)^{-2A}, \\
\|\theta \pm \frac{\pi}{2}\| \geq (\log H)^{-2A}
\end{array} \right\}.
\end{equation}
The area of $D_\cc$ satisfies $|D_\cc| \leq x^2$. Moreover, by (\ref{theta-H}), we know
\begin{align}\label{bound-D}
	\nonumber |D_\cc| & \gg \pi \left((r_0 + \eta^2x)^2 - (r_0 - \eta^2x)^2 \right) \cdot (\log H)^{-2A} \\
	& = 4\pi \eta^2 \left(\frac{1}{2}\tilde{H}/|g_\cc(\theta_\cc)|\right)^{1/d} x^2(\log H)^{-2A} \\
	\nonumber & \gg_\eta x^2(\log H)^{-2A}.
\end{align}

Recall from (\ref{def-lcf}) the definition of the localized counting function. By Lemma \ref{lbd-arc}, we have
\begin{align}
	\#\{ \cc \in S^{\mathrm{loc}}(H, \tilde{H}): \hat{N}_\cc(x) < \delta x^2 \} \leq \#\{ \cc \in S^{\mathrm{loc}}(H, \tilde{H}): \tilde{N}_\cc(D_\cc) \ll \delta x^2 \},
\end{align}
where we define
$$\tilde{N}_\cc(D_\cc) = \sum_{(m, n) \in {D_\cc \cap \ZZ^2}} \rho_K(W, g_\cc(m,n)).$$

\subsection{Results on binary forms}

In this section, we collect several results concerning binary forms that will be needed in the subsequent analysis. For a vector $\cc \in (c_0, \dots, c_d) \in \ZZ^{d+1}$, we define its \emph{content} by 
$$h_\cc = \gcd(c_0, \dots, c_d).$$
We define the content of a binary form $g_\cc$ in terms of the content of its coefficient vector.

We recall the definition (\ref{def-W}) of the modulus $W$. As in \cite[Section 9]{BST}, the very small prime factors of $W$ or the factors of $W$ that share a common factor with $h_\cc$ need to be treated separately.
Let $M_{d, K}$ be a sufficiently large positive constant that only depends on $d$ and the number field $K$. For any $\cc \in S^{\mathrm{loc}}(H, \tilde{H})$, we define
\begin{align}\label{def-W0}
	W_0 = \prod_{\substack{p \leq M_{d,K} \text{ or } p \mid h_\cc}} p^{k(x)}
\quad \text{and} \quad
W_1 = \prod_{\substack{M_{d,K} < p \leq w \\ p \nmid h_\cc}} p^{k(x)}.
\end{align}
We also introduce 
\begin{align}\label{def-sigma}
	\sigma_\cc(W_0) = \frac{1}{W_0^{e+1}} \#\{(\xx, s, t) \in (\ZZ / W_0 \ZZ)^{e+2}: \NNN_K(\xx) \equiv g_\cc(s,t) \, (\mod W_0)\}.
\end{align}
Moreover, we set
\begin{align}\label{def-Pc}
	P_\cc(w) = \prod_{\substack{M_{d, K} < p \leq w \\ p \nmid h_\cc}} p.
\end{align}

Let \(g\in \ZZ[s,t]\) be a binary form of degree \(d\). We say that $g$ is \emph{separable}, if the discriminant of $g$ is non-zero. Let
\begin{align}
	\lambda_g(q) = \# \{(u, v) \in (\ZZ / q\ZZ)^2: g(u, v) \equiv 0 \, (\mod q) \},
\end{align}
for any integer $q \geq 1$.
We now prove an upper bound for $\lambda_g(p^k)$ when $g$ is separable.
\begin{lemma}\label{bound-lambda}
	Let $p$ be a prime and let $k \in \NN$. Let $g \in \ZZ[s, t]$ be a binary form of degree $d \geq 1$. Assume that $g$ is separable. Let $h_g$ be the content of $g$ and let $\sigma = v_p(h_g).$
	\begin{enumerate}
		\item If $\sigma \geq k$ then $\lambda_g(p^k) = p^{2k}$.
		\item If $\sigma < k$ then
	$$\lambda_g(p^k) \leq (d+1) \min \left\{(k+1) p^{k(2 - \frac{1}{d}) + \frac{\sigma}{d}}, p^{2k-1} \right\}. $$ 
	\end{enumerate}
\end{lemma}
\begin{proof}
    The first statement is trivial. Now we assume that $\sigma < k$. 
    For a fixed integer $n$, we denote
    \begin{align}\label{gnm}
    	g_n(m) = g(m, n) = c_0m^d + (c_1n)m^{d-1} + \dots + (c_{d-1}n^{d-1})m + c_d n^d,
    \end{align}
    as a polynomial in $m$ with coefficients depending on $n$. Note that when $n$ is non-zero, the polynomial $g_n$ is separable if and only if the binary form $g$ is separable.
	We know
	\begin{align}\label{hgn}
		h_{g_n} = \gcd(c_0, c_1n, \dots, c_dn^d) \leq h_g \cdot n^d.
	\end{align}
	Let $\sigma_n = v_p(h_{g_n})$. According to \cite[Lemma 6.1]{BST}, we know $\lambda_{g_n}(p^k) = p^k$ when $\sigma_n \geq k$. If $\sigma_n < k$, then
	\begin{align}\label{BST-Lemma-6.1}
		\lambda_{g_n}(p^k) \leq d \min \left\{ p^{k(1 - \frac{1}{d} + \frac{\sigma_n}{d})}, p^{k-1} \right\}.
	\end{align}
	By using (\ref{hgn}) and the first bound of (\ref{BST-Lemma-6.1}), we obtain
	\begin{align*}
		\lambda_g(p^k) & = \sum_{1 \leq n \leq p^k} \lambda_{g_n}(p^k) \\
		& \leq d\sum_{1 \leq n \leq p^k} p^{k(1-1/d) + \sigma_n/d} \\
		& \leq d\sum_{1 \leq n \leq p^k} p^{k(1-1/d) + \sigma_n/d + v_p(n)} \\
		& \leq d \sum_{0 \leq v_p(n) \leq k} p^{k-v_p(n)} \cdot p^{k(1 - 1/d) + \sigma_n/d + v_p(n)} \\
		& = d \sum_{0 \leq v \leq k} p^{k + k(1 - 1/d) + \sigma_n/d} \\
		& \leq d(k+1) p^{k(2 - \frac{1}{d}) + \frac{\sigma_n}{d}}.
	\end{align*}
	By using the second bound of (\ref{BST-Lemma-6.1}), we know
	\begin{align*}
		\lambda_g(p^k) & = \sum_{1 \leq n \leq p^k} \lambda_{g_n}(p^k) \\
		& = \sum_{\substack{1 \leq n \leq p^k \\ p | n}} \lambda_{g_n}(p^k) + \sum_{\substack{1 \leq n \leq p^k \\ p \nmid n}} \lambda_{g_n}(p^k) \\
		& \leq p^{k-1} \cdot p^{k} + p^k \cdot d p^{k-1} = (d+1)p^{2k-1}.
	\end{align*}
	Combining these two estimates completes the proof
\end{proof}

Next, let $\tau$ be the divisor function.
We shall make frequent use of the following bound on the average size of $\tau$ evaluated at binary form arguments. 

\begin{lemma}\label{poly-tau}
Let $\delta>0$ and $C>0$. Let $g \in \ZZ[s, t]$ be a separable binary form of degree $d\geq 1$. Let $\|g\|$ be the maximum modulus of the coefficients of $g$ and assume that the content $h_g$ of $g$ is at most $(\log x)^C$. Then there exist a positive constant $K$, depending only on $\delta$, $d$ and $C$, such that for $x \geq \|g\|^{\delta}$ we have
$$
\sum_{\substack{m, n\leq x}} \tau(|g(m, n)|)^C \ll x^2 (\log x)^{K}.
$$
\end{lemma}
\begin{proof}
	We use the notation (\ref{gnm}). Let $h_{g_n}$ be the content of $g_n$ so that $g_n(m) = h_{g_n} \hat{g}_n(m)$, where $\hat{g}_n$ has content $1$. By \cite[Lemma 6.2]{BST}, we know
	$$\sum_{m \leq x} \tau(|g_n(m)|)^C \ll x(\log x)^{K'},$$
	for some $K' > 0$, depending only on $\delta, d$ and $C$. 
	By using the assumption $h_g \leq (\log x)^C$, we know $\tau(h_g) \ll h_g^{1/C} \ll \log x$. We then apply (\ref{sub-tau}) and (\ref{hgn}) to obtain
	\begin{align*}
		\sum_{m, n \leq x} \tau(|g(m, n)|)^C & = \sum_{n \leq x} \sum_{m \leq x} \tau(|g_n(m)|)^C  \\
		& \leq \sum_{n \leq x} \tau(h_{g_n})^C \sum_{m \leq x} \tau(|\hat{g}_n(m)|)^C \\
		& \ll x(\log x)^{K'} \sum_{n \leq x} \tau(h_g)^C \tau(n^d)^C \\
		& \ll x(\log x)^{K' + 1} \sum_{n \leq x} \tau(n)^{dC} \\
		& \ll x(\log x)^K,
	\end{align*}
	for some $K > 0$, depending only on $\delta, d$ and $C$. 
\end{proof}

\subsection{A transition to non-archimedean densities}
Let $K$ be a number field with ring of integers $\mathfrak{o}_K$. For an integral ideal $\mathfrak{a} \subset \mathfrak{o}_K$ we define its norm by
$$\mathrm{N}\mathfrak{a}=\# (\mathfrak{o}_K/\mathfrak{a}).$$
The \emph{Dedekind zeta function} associated to $K$ is defined for $\mathrm{Re}(s)>1$ by
$$
\zeta_K(s)=\sum_{\substack{\mathfrak{a} \subset \mathfrak{o}_K\\\mathfrak{a}\neq (0)}} \frac{1}{(\mathrm{N} \mathfrak{a})^s} =
\sum_{m=1}^\infty \frac{r_K(m)}{m^{s}},
$$
where
$$
r_K(m)=\#\{\mathfrak{a} \subset \mathfrak{o}_K: \mathrm{N} \mathfrak{a}=m\}.
$$
It is well known that $r_K(m)$ is multiplicative and that $\zeta_K(s)$ admits a meromorphic continuation to the entire complex plane with a simple pole at $s=1$. 

For any prime $p$ and any integer $k \geq 1$, we define
$$
\alpha_p = \Big( 1 - \frac{1}{p} \Big)^{-1} \prod_{\mathfrak{p} | p} \Big( 1 - \frac{1}{\mathrm{N} \mathfrak{p}} \Big)
$$
and 
$$
\beta_{p^k} = p^k \sum_{j \geq k} \frac{r_K(p^j)}{p^j} \prod_{\mathfrak{p} | p} \left( 1 - \frac{1}{\mathrm{N} \mathfrak{p}} \right).
$$
Next, define the multiplicative function $\tilde{r}_K(n)$ on prime powers, with $k(x)$ as in (\ref{def-kx}), by
\[
\tilde{r}_K(p^j) = 
\begin{cases} 
r_K(p^j), & \text{if } j < k(x), \\
\alpha_p^{-1} \beta_{p^{k(x)}}, & \text{if } j \geq k(x).
\end{cases}
\]
Define also $b = \tilde{r}_K * \mu$. Then \cite[Formula (7.27)]{BST} shows that
\begin{align}\label{bound-g}
	|b(k)| \leq \tau(k)^{e+1}.
\end{align}

We now state the main result of this section.
\begin{prop}\label{prop-na}
	Let $A, C \geq 1$ be fixed. Take $B = \tilde{H}^{1/e} x^{d/e}$ and assume (\ref{def-tilde-H}) holds. For $\epsilon \in \{\pm\}$ and $\delta = (\log x)^{-C}$,  there exists a constant $M = M_{C, d, K} \geq 1$ such that
\begin{align*}
& \#\left\{ \cc\in S_\epsilon^{\textnormal{loc}}(\tilde H, H): 
\hat N_\cc(x) <\delta x^2\right\} 
 \ll \frac{H^{d+1}}{(\log x)^{dC}}\\
& \quad \quad \quad \quad \quad \quad \quad \quad \quad \quad \quad \ + \#
\left\{ \cc\in S_\epsilon^{\textnormal{loc}}(\tilde H, H):  \sigma_{\cc}(W_0) \leq \delta M (\log x)^{c_{d,e}} (\log H)^{2A}\right\},
\end{align*}
where $c_{d,e} = e + d^2(d+1)^{e+2}$.
\end{prop}

The following class of binary forms plays an important role in the subsequent analysis.
\begin{definition}[\(C\)-admissible binary forms]\label{def-C-adm}
	Let \(g\in \ZZ[s,t]\) be a binary form of degree \(d\). Given a constant \(C\ge 1\), we say that \(g\) is \emph{\(C\)-admissible} if the following conditions hold:
\begin{itemize}
	\item $g$ is separable.
	\item $g$ has no integer zero, that is, for any $(m, n) \in \ZZ^2 \setminus \{(0,0)\}$, we have $g(m, n) \neq 0$.
	\item $g$ has content $\leq (\log x)^C$.
\end{itemize}
\end{definition}
The following result is the key step in the proof of Proposition \ref{prop-na}.
\begin{lemma}\label{BST-L-9.3}
	Let $C \geq 1$ and let $\cc \in S(H, \tilde{H})$. Suppose that $g_\cc$ is $C$-admissible. Then
	\begin{align*}
		\tilde{N}_\cc(D_\cc) = \sigma_\cc(W_0)|D_\cc| \sum_{\substack{k \in \mathbb{N} \\ p \mid k \Rightarrow p \mid P_\cc(w)}} 
\frac{c_\cc(w) b(k) \lambda_{g_\cc}(k)}{k^2} 
+ O_C \left( x^2 \exp \left( - (\log x)^{1/8} \right) \right),
	\end{align*} 
	where $D_\cc$ is given by (\ref{def-Dc}), $P_\cc(w)$ is given by (\ref{def-Pc}) and
	$$c_\cc(w) = \prod_{p | P_\cc(w)} \alpha_p.$$
\end{lemma}

\begin{proof}
	By Definition \ref{def-C-adm}, the content of $g_\cc$ is $\leq (\log x)^C$. Recall from  (\ref{def-W0}) the definition of $W_0$. It follows from \cite[Formula (9.4)]{BST} that
	\begin{align}\label{bound-W0}
		W_0 \leq \exp\big(C'(\log \log x)^3\big),
	\end{align}
	where $C'$ is a constant only depends on $d, K$ and $C$.
    According to \cite[Formula (9.5)]{BST}, we may write
	\begin{align}\label{BST-9.5}
		\tilde{N}_\cc(D_\cc) = \sum_{\nu_1, \nu_2 \, (\mod W_0)} \rho_K(W_0, g_\cc(\nu_1, \nu_2))T(D_\cc; \nu_1, \nu_2),
	\end{align} 
	where 
	\begin{align}\label{def-T}
		T(D_\cc; \nu_1, \nu_2) = \sum_{\substack{(m,n) \in D_\cc \\ (m, n) \equiv (\nu_1, \nu_2) \, (\mod W_0)}} \rho_K(W_1, g_\cc(m, n)).
	\end{align}
	A further application of \cite[Formula (9.6)]{BST} gives
	\begin{align}\label{BST-9.6}
		T(D_\cc; \nu_1, \nu_2) = U(D_\cc; \nu_1, \nu_2) \prod_{p | P_\cc(w)} \alpha_p,
 	\end{align}
 	where
 	$$U(D_\cc; \nu_1, \nu_2) = \sum_{\substack{(m,n) \in D_\cc \\ (m,n) \equiv (\nu_1, \nu_2) \, (\mod W_0)}} \tilde{r}_K(g_\cc(m,n)_W),$$
 	with the notation 
 	$$n_W = \prod_{\substack{p^\nu \mid\mid n \\ p \mid P(w)}} p^\nu.$$
 	
We know  $g_\cc(m, n) \neq 0$ for all $(m, n) \in D_\cc \cap \ZZ^2$. Write
$$\mathcal{C}_w = \{n \in \NN: p | n \, \Rightarrow \, p | P_\cc(w)\}.$$
Let $d' = g_\cc(m, n)_W$. Then there exists an integer $d''$ such that $g_\cc(m, n) = d'd''$, with $d' \in \mathcal{C}_w$ and $\gcd(d'', P_\cc(w)) = 1$. An application of M\"obius inversion then yields
$$U(D_\cc; \nu_1, \nu_2) = \sum_{k \in \mathcal{C}_w} b(k) \sum_{\substack{(m, n) \in D_\cc \\ (m, n) \equiv (\nu_1, \nu_2) \, (\mod W_0) \\ k | g_{\cc}(m, n)}} 1,$$
where $b = \tilde{r}_K * \mu$. Then we break the outer sum into two ranges, leading to
\begin{align}\label{BST-9.7}
	U(D_\cc; \nu_1, \nu_2) = \Sigma_1 + \Sigma_2,
\end{align}
where
$$\Sigma_1 = \sum_{\substack{k \in \mathcal{C}_w \\ k \leq x^{1/(2d)}}} b(k) \sum_{\substack{(m, n) \in D_\cc \\ (m, n) \equiv (\nu_1, \nu_2) \, (\mod W_0) \\ k | g_{\cc}(m, n)}} 1$$
and
$$\Sigma_2 = \sum_{\substack{k \in \mathcal{C}_w \\ k > x^{1/(2d)}}} b(k) \sum_{\substack{(m, n) \in D_\cc \\ (m, n) \equiv (\nu_1, \nu_2) \, (\mod W_0) \\ k | g_{\cc}(m, n)}} 1 .$$

For the contribution $\Sigma_2$, by (\ref{bound-g}), we have
$$\Sigma_2 \ll \sum_{\substack{k \in \mathcal{C}_w \\ k > x^{1/(2d)}}} \tau(k)^{e+1} \sum_{\substack{m, n \leq x \\ k|g_\cc(m,n)}}1.$$
Suppose that $\omega(k) = r$. Then any $k$ in the sum satisfies
$$x^{1/(2d)} < k = p_1^{j_1} \cdots p_r^{j_r} \leq w^{k(x)r},$$
hence $\omega(g_\cc(m, n)) \geq \omega(k) = r > K_0$, with
$$K_0 = \frac{\log x}{2k(x) d \log w} = \frac{\sqrt{\log x}}{2k(x)d}.$$
It then follows that $1 \leq 2^{-K_0}2^{\omega(g_\cc(m, n))} \leq 2^{-K_0} \tau(|g_\cc(m, n)|)$, so
$$\Sigma_2 \ll \frac{1}{2^{K_0}} \sum_{m,n\leq x} \tau(|g_\cc(m, n)|)^{e+2}.$$
Since $g_\cc$ has content $\leq (\log x)^C$, it now follows from Lemma \ref{poly-tau} that
\begin{align}\label{sum_2}
	\Sigma_2 \ll_A \frac{x(\log x)^{K_1}}{2^{K_0}} \ll x \exp\left(-(\log x)^{1/4} \right), 
\end{align} 
for a suitable constant $K_1$ only depending on $d, e$ and $A$.

For the contribution $\Sigma_1$, recall that the region $D_\cc$ compact and is the union of at most two compact regions, each formed by the difference of two convex regions.
% = D_\cc^{(1)} \setminus D_\cc^{(2)}$, where 
%$$D_\cc^{(1)} = \big\{(u, v) \in [-x, x]^2: 0 < r  < r_0 + \eta^2 x,  \norm{\theta - \theta_\cc} \leq (\log H)^{-2A}  \big\};$$
%$$D_\cc^{(2)} = \big\{(u, v) \in [-x, x]^2: 0 < r  < r_0 - \eta^2 x,  \norm{\theta - \theta_\cc} \leq (\log H)^{-2A}  \big\}$$
%are two convex regions.
By \cite[Lemma 1]{Sch}, we obtain
$$\Sigma_1 = \sum_{\substack{k \in \mathcal{C}_w \\ k \leq x^{1/(2d)}}} b(k) \lambda_{g_\cc}(k) \bigg( \frac{|D_\cc|}{W_0^2 k^2} + O(x) \bigg).$$
From (\ref{bound-g}), we know $|b(k)| = O_\eps(k^\eps)$. We use the trivial bound $\lambda_{g_\cc}(k) \leq k^2$ to obtain
    \begin{align}\label{BST-9.10}
    	\Sigma_1 = \frac{|D_\cc|}{W_0^2} \sum_{\substack{k \in \mathcal{C}_w \\ k \leq x^{1/(2d)}}} \frac{b(k) \lambda_{g_\cc}(k)}{k^2} + O_{A, \eps}(x^{1 + \frac{3}{2d} + \eps}).
    \end{align}
    Since $|D_\cc| \leq x^2$, we see that
    \begin{align*}
    	\frac{|D_\cc|}{W_0^2}  \sum_{\substack{k \in \mathcal{C}_w \\ k \leq x^{1/(2d)}}} \frac{b(k) \lambda_{g_\cc}(k)}{k^2} = \frac{|D_\cc|}{W_0^2}  \sum_{\substack{k \in \mathcal{C}_w}} \frac{b(k) \lambda_{g_\cc}(k)}{k^2} + O(x^2 \Upsilon(x)),
    \end{align*}
    where
    \begin{align}
    	\Upsilon(x) = \sum_{\substack{k \in \mathcal{C}_w \\ k > x^{1/(2d)}}} \frac{|b(k)| \lambda_{g_\cc}(k)}{k^2}.
    \end{align}
    
    We now claim that
    \begin{align}\label{bound-Gamma}
    	\Upsilon(x) \ll x^{-1/(4d)}.
    \end{align}
    To prove this, we write $k = k_1k_2$ for coprime $k_1, k_2$, where $k_1$ is square-free and $k_2$ is square-full. Moreover, we may assume that $v_p(k_2) \leq k(x)$ for any prime $p | k_2$, since otherwise $b(k) = 0$. Furthermore any $p | k_1k_2$ is coprime to the content of $g_\cc$, since $k \in \CCC_w$. Note that $\lambda_{g_\cc}(k) = \lambda_{g_\cc}(k_1) \lambda_{g_\cc}(k_2)$ and
    $$\lambda_{g_\cc}(k_1) \leq \prod_{p \| k_1 }(d+1) \leq (d+1)^{\omega(k_1)} \leq \tau(k_1)^{d+1},$$
    by Lemma \ref{bound-lambda}. Then, combining this with (\ref{bound-g}), we deduce that
    $$\Upsilon(x)\ll \sum_{\substack{k_2\in \mathcal{C}_w\\p^\nu\| k_2\Rightarrow 2\leq \nu\leq k(x)}} \frac{\tau(k_2)^{e+1}\lambda_{g_\cc}(k_2)}{k_2^2} \Sigma_{d+e+2}\left(\frac{x^{1/(2d)}}{k_2}\right),$$
where 
$$
\Sigma_B(z)=\sum_{\substack{p\mid k \Rightarrow p\leq w\\ k>z}} \frac{\tau(k)^{B}}{k^2} \ll_B \sum_{k > z} \frac{k^\eps}{k^2} \ll z^{-1+\eps}, $$
for any $\eps > 0$.
In particular, 
$$\Sigma_{d+e+2}\left(\frac{x^{1/(2d)}}{k_2}\right) \ll \frac{k_2}{x^{1/(2d)-\eps}} ,$$
since $\log w = \sqrt{\log x}$. 
Hence we obtain
\begin{align*}
	\Upsilon(x) & \ll \ \frac{k_2}{x^{1/(2d) - 2\eps}} 
 \sum_{\substack{k_2\in \mathcal{C}_w\\
p^\nu\| k_2\Rightarrow 2\leq \nu\leq k(x)}}
 \frac{\tau(k_2)^{e+1}
\lambda_{g_\cc}(k_2)}{k_2^2} \\
 & = \ \frac{1}{x^{1/(2d) - 2\eps}} 
 \sum_{\substack{k_2\in \mathcal{C}_w\\
p^\nu\| k_2\Rightarrow 2\leq \nu\leq k(x)}}
 \frac{\tau(k_2)^{e+1}
\lambda_{g_\cc}(k_2)}{k_2}.
\end{align*}
It follows from multiplicativity and Lemma \ref{bound-lambda} that
\begin{align*}
	\sum_{\substack{k_2\in \mathcal{C}_w\\
p^\nu\| k_2\Rightarrow 2\leq \nu\leq k(x)}}
 \frac{\tau(k_2)^{e+1}
\lambda_{g_\cc}(k_2)}{k_2} & \leq \prod_{p \leq w} \left( 1 + \sum_{2 \leq j \leq d} \frac{d(j+1)^{e+2}}{p} + \sum_{j \geq d+1} \frac{d(j+1)^{e+2}}{p^{j/d}} \right) \\
& = \prod_{p \leq w} \left( 1 + O \left( \frac{1}{p} \right) \right) \\
& \ll (\log x)^B.
\end{align*}
This proves (\ref{bound-Gamma}).
    
By (\ref{bound-Gamma}), it follows that
    \begin{align*}
    	\Sigma_1 = \frac{|D_\cc|}{W_0^2}  \sum_{\substack{k \in \mathcal{C}_w}} \frac{b(k) \lambda_{g_\cc}(k)}{k^2} + O_A\left(x^{2 - 1/(4d)}\right),
    \end{align*}
    in (\ref{BST-9.10}). Combining this with (\ref{sum_2}) in (\ref{BST-9.7}), we obtain
    \begin{align*}
    	U(D_\cc; \nu_1, \nu_2) = \frac{|D_\cc|}{W_0^2}  \sum_{\substack{k \in \NN \\ p | k \Rightarrow p | P_\cc(w)}} \frac{b(k) \lambda_{g_\cc}(k)}{k^2} + O_A\left(x^2 \exp (-(\log x)^{1/4})\right).
    \end{align*}
    We now insert this into (\ref{BST-9.6}) to obtain
    \begin{align}\label{prop-T}
    	T(D_\cc; \nu_1, \nu_2) = \frac{|D_\cc|}{W_0^2}  \sum_{\substack{k \in \NN \\ p | k \Rightarrow p | P_\cc(w)}} \frac{c_\cc(w)b(k) \lambda_{g_\cc}(k)}{k^2} + O_A\left(x^2 \exp (-(\log x)^{1/4})\right).
    \end{align}
    We further substitute this into (\ref{BST-9.5}), noting that $c_\cc(w) \leq W/\varphi(W) \ll \sqrt{\log x}$ and
    \begin{align*}
    	\sum_{\nu_1, \nu_2 \, (\mod W_0)} \frac{\rho_K(W_0, g_\cc(\nu_1, \nu_2))}{W_0^2} = \sigma_\cc(W_0),
    \end{align*}
    in the notation of (\ref{def-sigma}). The main term therefore agrees with the statement of the lemma. By (\ref{bound-W0}), we see
    $$c_\cc(w) x^2 \exp\left(-(\log x)^{1/4}\right) \sum_{\nu_1, \nu_2 \, (\mod W_0)} \frac{\rho_K(W_0, g_\cc(\nu_1, \nu_2))}{W_0^2} \ll x^2 \exp\left(-(\log x)^{1/8}\right).$$
    This completes the proof of the lemma.
\end{proof}

\begin{lemma}\label{BST-L-9.4}
	Let $g \in \ZZ[s, t]$ be a separable binary form of degree $d$ and content $h$. Define
	$$P = \prod_{\substack{M_{d,K} < p \leq w \\ p \nmid h}} p.$$
	Then we have
	$$\sum_{\substack{k \in \mathbb{N} \\ p \mid k \Rightarrow p \mid P}} 
\frac{c_\cc(w) b(k) \lambda_{g}(k)}{k^2} \gg (\log x)^{-c_{d, e}}, $$
    where $c_{d, e} = e + d^2(d+1)^{e+2}$.
\end{lemma}
\begin{proof}
	By multiplicativity, we may write 
	$$\sum_{\substack{k \in \mathbb{N} \\ p \mid k \Rightarrow p \mid P}}  \frac{c_\cc(w) b(k) \lambda_{g}(k)}{k^2} = \prod_{p | P} \alpha_p \xi_p,$$
    where
    \begin{align*}
    	\xi_p & = 1 + \sum_{j \geq 1} \frac{b(p^j)\lambda_{g}(p^j)}{p^{2j}}.
    \end{align*}
Using the bound (\ref{bound-g}) and Lemma \ref{bound-lambda}, we obtain
\begin{align*}
    	\xi_p & \geq 1 - \sum_{j \geq 1} \frac{d (j+1) \tau(p^j)^{e+1} \min (p^{2j-1}, p^{j(2 - 1/d)})}{p^{2j}} \\
    	& = 1 - \sum_{j = 1}^d \frac{d(j+1)^{e+2}}{p} - \sum_{j \geq d+1} \frac{d(j+1)^{e+2}}{p^{j/d}} \\
    	& \geq 1 - \sum_{j = 1}^d \frac{d(j+1)^{e+2}}{p} - \sum_{j \geq d+1} \frac{d(j+1)^{e+2}}{p^{j/d}},
    \end{align*}
On the other hand, we also have 
$$\alpha_p \geq 1 - \frac{e-1}{p} + O \left( \frac{1}{p^2} \right),$$
by \cite[Formula (6.8)]{BST}.
Thus,
$$\alpha_p \xi_p \geq 1 - \frac{e + d^2(d+1)^{e+2}}{p} + O \left( \frac{1}{p^2} \right).$$
On assuming that $M_{d, K}$ is sufficiently large, the asserted lower bound follows from Mertens' theorem.

\end{proof}

\begin{proof}[Proof of Proposition \ref{prop-na}]
By the argument in the proof of \cite[Proposition 9.1]{BST}, the number of coefficient vectors $\cc$ for which $g_\cc$ either fails to be separable or has content greater than $(\log x)^C$ is 
$\ll H^{d+1}(\log x)^{-dC}.$
Suppose $g_\cc(m, n) = 0$ for non-zero $(m, n) \in \ZZ^2$. Without loss of generality, we assume $n \neq 0$. Then the polynomial
$$f_\cc(x) = c_0 x^d + c_1 x^{d-1} + \dots + c_d$$
has a rational solution $x = m/n$. According to \cite{Kuba}, there are $O(H^d \log H)$ choices of $\cc \in S^{\mathrm{loc}}(H, \tilde{H})$ for which $f_\cc$ is reducible over $\QQ$. It follows that there are $O(H^{d}\log H)$ choices of $\cc \in S^{\mathrm{loc}}(H, \tilde{H})$ for which $f_\cc$ has a rational root, since these are the polynomials which admit a linear factor over $\QQ$. Therefore, we know 
$$\# \left\{ \cc \in S^{\mathrm{loc}}(H, \tilde{H}): g_\cc \text{ is not $C$-admissible} \right\} \ll \frac{H^{d+1}}{(\log x)^{dC}}.$$

On the other hand, when $g_\cc$ is $C$-admissible, we have $$\tilde{N}_\cc(D_\cc) \gg \sigma_\cc(W_0) \frac{x^2}{(\log x)^{c_{d,e}} (\log H)^{2A}} + O \left( x^2 \exp\left( - (\log x)^{1/8} \right) \right),$$
by (\ref{bound-D}), Lemma \ref{BST-L-9.3} and Lemma \ref{BST-L-9.4}. The statement of Proposition \ref{prop-na} now follows.
\end{proof}

\subsection{The non-archimedean densities are rarely small}
Recall the definition (\ref{def-W0}) of $W_0$.
In this section, we establish an upper bound for
\begin{align*}
	M(H, \tilde{H}; \Delta) = \#\{\cc \in S^{\mathrm{loc}}(H, \tilde{H}): \sigma_\cc(W_0) < \Delta \}.
\end{align*}
Let $U(W_0) \subset (\ZZ / W_0 \ZZ)^{d+1}$ be the image of the set $S^{\mathrm{loc}}(H, \tilde{H})$ under reduction modulo $W_0$. By (\ref{bound-W0}), we know $W_0 \leq \tilde{H} \leq H$. By partitioning according to the congruence class modulo $W_0$, we have
\begin{align*}
	M(H, \tilde{H}; \Delta) & \leq \sum_{\substack{\uu \in U(W_0) \\ \sigma_\uu(W_0) < \Delta}} \# \{ \cc \in S^{\mathrm{loc}}(H, \tilde{H}): |\cc| \leq H, \cc \equiv \uu \, (\mod W_0) \} \\
	& \ll \sum_{\substack{\uu \in U(W_0) \\ \sigma_\uu(W_0) < \Delta}} \frac{\tilde{H}}{W_0} \Big( \frac{H}{W_0} \Big)^d.
\end{align*}
It follows from Rankin's trick that
\begin{align*}
	M(H, \tilde{H}; \Delta) \ll \frac{\Delta^\kappa \tilde{H} H^d}{W_0^{d+1}} \sum_{\uu \in U(W_0)} \frac{1}{\sigma_\uu(W_0)^\kappa},
\end{align*}
for any $\kappa > 0$.
By multiplicativity, we may write
\begin{align}\label{final}
	M(H, \tilde{H}; \Delta) \ll \frac{\Delta^\kappa \tilde{H} H^d}{W_0^{d+1}} \prod_{p | W_0} \sum_{\uu \in U(p^{k(x)})} \frac{1}{\sigma_\uu(p^{k(x)})^\kappa},
\end{align}

The next result provides a good lower bound for $\sigma_\uu(p^k)$ whenever the equation $\NNN_K(\xx) = g_\uu(s,t)$ has a $p$-adic solution which is not too singular. It is the binary form analogue of \cite[Lemma 9.5]{BST} and the proof is similar.

\begin{lemma}\label{hensel}
	Let $p^k$ be a prime power and let $\uu \in (\ZZ / p^k \ZZ)^{d+1}$. Assume that there exists an integer $\alpha \geq 0$ and a solution $(\xx_0, s_0, t_0) \in \ZZ_p^{e+2}$ satisfying
	\begin{align}\label{eq:goat}
		\NNN_K(\xx_0) \equiv g_\uu(s_0, t_0) \, \mod p^k \text{ \indent  and \indent } p^\alpha \| (\nabla \NNN_K(\xx_0), \nabla g_\uu(s_0, t_0)).
	\end{align}
	Then $\sigma_\uu(p^k) \geq p^{-(\alpha+1)(e+1)}$.
\end{lemma}

We now state our final bound for $M(H, \tilde{H}; \Delta)$.

\begin{prop}\label{BST-P-9.6}
	We have 
	$$M(H, \tilde{H}; \Delta) \ll \Delta^{\frac{1}{d^2(e+1)}} \tilde{H} H^d.$$
\end{prop}
\begin{proof}
	We stratify the set $U(p^{k(x)})$ according to the $p$-adic valuation $\alpha$ of the vector 
$$ (\nabla \NNN_K(\xx_0), \nabla g_\uu(s_0, t_0)).$$
For $0\leq \alpha\leq k(x)$, let $U_\alpha(p^{k(x)})$ be the set of $\uu \in (\ZZ/p^{k(x)}\ZZ)^{d+1}$ such that  $p\nmid \uu$, and for which there exists
$ (\xx_0,s_0, t_0)\in \ZZ_p^{e+1}$ such that \eqref{eq:goat} holds with $k=k(x)$.
Note that $U_{k(x)}(p^{k(x)})$ should actually be defined with the condition that 
$$p^{k(x)} | (\nabla \NNN_K(\xx_0), \nabla g_\uu(s_0, t_0))$$ in~\eqref{eq:goat}. Then we have
$$
U(p^{k(x)})=U_0(p^{k(x)})\sqcup U_1(p^{k(x)})\sqcup \cdots 
\sqcup
U_{{k(x)}}(p^{k(x)}).
$$
From Lemma \ref{hensel}, we know
\begin{align*}
\sum_{\uu\in U(p^{k(x)})} \frac{1}{\sigma_\uu(p^{k(x)})^\kappa}
\leq \sum_{0\leq \alpha\leq {k(x)}} 
p^{(\alpha+1)(e+1) \kappa} \cdot \#U_\alpha(p^{k(x)}).
\end{align*}

According to the proof of \cite[Proposition 9.6]{BST}, we know 
$$\#U_\alpha(p^{k(x)}) \ll p^{k(x) (d+1) - \frac{\alpha}{d(d-1)}}.$$
Thus, summing over $\alpha$ we deduce
$$\sum_{\uu \in U(p^{k(x)})} \frac{1}{\sigma_\uu(p^{k(x)})^\kappa} \ll \sum_{0 \leq \alpha \leq k(x)} p^{(\alpha+1)(e+1)\kappa} \cdot p^{k(x)(d+1) - \frac{\alpha}{d(d-1)}}.$$
If we choose
$$\kappa = \frac{1}{d^2(e+1)},$$
then $e \kappa < \frac{1}{d(d-1)}$. Note that $p|W_0 = O(1)$, we obtain
$$\sum_{\uu \in U(p^{k(x)})} \frac{1}{\sigma_\uu(p^{k(x)})^\kappa} \ll p^{k(x)(d+1)}.$$
Making this choice of $\kappa$, we return to (\ref{final}) and therefore arrive at the statement of the proposition.
\end{proof}

\begin{proof}[Proof of Proposition \ref{lcf-rs}]
	This follows from Propositions \ref{prop-na} and \ref{BST-P-9.6}.
\end{proof}

\bibliographystyle{amsplain}
\bibliography{Ref}

\end{document}